\documentclass[12pt,reqno]{amsart} 
\usepackage{setspace}
\usepackage{amsmath}\allowdisplaybreaks
\usepackage{amscd,amssymb,stmaryrd,amstext,amsthm}
\usepackage{url}
\usepackage{subfigure}
\usepackage{enumitem}
\usepackage{graphicx}
\usepackage{color}
\usepackage{epstopdf}
\usepackage{caption}
\usepackage{multicol}
\usepackage{tikz}
\usepackage{pgfplots} 
\pgfplotsset{compat=newest}
\usetikzlibrary{decorations.pathreplacing,angles,quotes}
\usetikzlibrary{bending}
\usepackage{mathtools}
\usepackage[colorlinks]{hyperref}
\usepackage[all]{hypcap}
\usepackage[lite]{amsrefs}
\usepackage{amsfonts}
\usepackage[normalem]{ulem}

\newcounter{desccount}
\newcommand{\descitem}[1]{%
  \item[#1] \refstepcounter{desccount}\label{#1}
}
\newcommand{\descref}[1]{\hyperref[#1]{#1}}

\numberwithin{equation}{section}

\newcommand{\sub}{\subseteq}

\newcommand{\R}{\mathbb{R}}
\newcommand{\C}{\mathbb{C}}

\newcommand{\bs}{\backslash}

\newcommand{\Gd}{\delta}
\newcommand{\Ge}{\varepsilon}

\newcommand{\Gl}{\lambda}

\newcommand{\Gs}{\sigma}

\numberwithin{chap}{section}
\newtheorem{thm}{Theorem}
\numberwithin{thm}{section}
\newtheorem{conj}{Conjecture}
\numberwithin{conj}{section}
\newtheorem{prop}[thm]{Proposition}
\newtheorem{defn}[thm]{Definition}
\newtheorem{lem}[thm]{Lemma}

\DeclarePairedDelimiter{\norm}{\lVert}{\rVert}
\DeclarePairedDelimiter{\abs}{\lvert}{\rvert}
\DeclarePairedDelimiter{\ip}{\langle}{\rangle}

\makeatletter
\let\oldabs\abs
\def\abs{\@ifstar{\oldabs}{\oldabs*}}
\let\oldnorm\norm
\def\norm{\@ifstar{\oldnorm}{\oldnorm*}}
\let\oldip\ip
\def\ip{\@ifstar{\oldip}{\oldip*}}
\makeatother

\begin{document}

\pagestyle{myheadings} \thispagestyle{empty} \markright{}
\title{Decoupling for mixed-homogeneous polynomials in $\mathbb R^3$}

\author{Jianhui Li and Tongou Yang}
\address[Jianhui Li]{Department of Mathematics, University of Wisconsin-Madison, Van Vleck Hall, 480 Lincoln Drive, Madison, WI 53706-1325}
\email{jli2266@wisc.edu}

\address[Tongou Yang]{Department of Mathematics, University of British Columbia, Vancouver, B.C. Canada V6T 1Z2}
\email{toyang@math.ubc.ca}


\date{}
\maketitle

\begin{abstract}
     We prove decoupling inequalities for mixed-homogeneous bivariate polynomials, which partially answers a conjecture of Bourgain, Demeter and Kemp.  
\end{abstract}


\section{Introduction}

\subsection{Background}
Bourgain and Demeter's 2015 breakthrough $l^2$-decoupling theorem for the truncated elliptic paraboloid in $\R^n$ \cite{BD2015} is of substantial importance in harmonic analysis, and has been generalised in various directions since then. In \cite{BGLSX}, the authors studied $l^2$-decoupling inequalities in $\R^2$ for general real analytic phase functions over a compact interval. Following this, Demeter \cite{Demeter2020} improved upon \cite{BGLSX} by choosing partitions of the unit interval that fit the curvature of each analytic phase function. Later, the second author \cite{Yang2} further proved a uniform decoupling inequality for all polynomial phase functions with a given bound on the degree of the polynomials. By a standard Taylor polynomial approximation, this can be shown to imply decoupling for every single smooth function on a compact interval, generalising Demeter's result.

However, in higher dimensions, the problem becomes much harder. Even in $\R^3$, the decoupling with respect to an arbitrary real-analytic phase function over a compact set remains unknown:

\begin{conj}[Bourgain, Demeter and Kemp, \cite{BDK2019}]\label{conj_BDK2019}
Let $\phi:(-2,2)^2\to \R$ be a real analytic function. Then for every $\Ge>0$, there is a constant $C_\Ge$, depending on $\phi$ and $\Ge$ only, such that the following is true. For every $0<\Gd<1$, there is a boundedly overlapping family $\mathcal P=\mathcal P_\Gd$ of rectangles covering $[-1,1]^2$, such that $\phi$ is $\Gd$-flat\footnote{For a formal definition, see the appendix.} over each $P\in \mathcal P$, and for any function $f:\R^3\to \C$ with Fourier support on the $\Gd$-neighbourhood of the graph of $\phi$ above $[-1,1]^2$, we have the $l^4$-decoupling inequality:
\begin{equation}\label{eqn_decoupling_l4_conj}
    \norm {f}_{L^4(\R^3)}\leq C_\Ge\Gd^{-\Ge}\# \mathcal P^{\frac 1 4}\left(\sum_{P\in \mathcal P}\norm{f_P}_{L^4(\R^3)}^4\right)^{\frac 1 4}.
\end{equation}
\end{conj}
Here and throughout this paper, $f_P$ stands for the Fourier restriction of $f$ to the strip $P\times \R$, namely, $f_P$ is defined by the relation
$$
\widehat {f_P}(x,y,z)=\hat f(x,y,z)1_P(x,y).
$$
In \cite{BDK2019} the authors partially answered the conjecture when $\phi$ is given by a real analytic surface of revolution. Later Kemp \cite{Kemp} partially proved the conjecture by establishing a decoupling inequality for all surfaces with identically zero Gaussian curvature but no umbilical points. While this article was under review, we were informed that he also proved in \cite{Kemp2} decoupling inequalities for a broad class of $C^5$ surfaces in $\R^3$ lacking planar points.

Degenerate phases are also of interest in other related problems in harmonic analysis, including the restriction problem and the $L^p$ improving estimates for averages. For example, the reader may refer to \cites{M2014,IM2016,DZ2019,IM2011, PS1997,Oberlin2012,CKZ2013,SS2019}.

\subsection{Main result}
In this paper, we seek to partially answer the conjecture for a model class of polynomials, namely, real-valued mixed-homogeneous polynomials. A polynomial $\phi$ is said to be mixed-homogeneous if for some positive integers $q,r,s$ with $\mathrm{gcd}(r,s)=1$ we have
\begin{equation}\label{eqn_mixed_homo_def0}
   \phi(x,y) = \rho^{-q} \phi(\rho^r x, \rho^s y),\quad \text{for all $(x,y)\in \R^2$ and all $\rho>0$}. 
\end{equation}

The properties of mixed-homogeneous polynomials that will be used in this paper are listed in the appendix.

Mixed-homogeneous polynomials have played an important role in the development of a general theory for general polynomial phases, for instance \cites{DZ2019,IM2011,IM2016, PS1997}.


Now we are ready to state the main theorem of the paper.
\begin{thm}[Main theorem]\label{thm_decoupling}
Let $\phi:\R^2\to \R$ be a mixed-homogeneous polynomial. Then for every $\Ge>0$, there is a constant $C_\Ge=C_{\Ge,\phi}>0$ such that the following holds. For every $0<\Gd<1$, there is a boundedly overlapping family $\mathcal P=\mathcal P_\Gd$ of rectangles covering $[-1,1]^2$, such that $\phi$ is $\Gd$-flat over each $P\in \mathcal P$, and for any function $f:\R^3\to \C$ with Fourier support on the $\Gd$-neighbourhood of the graph of $\phi$ above $[-1,1]^2$, we have the $l^4(L^4)$-decoupling inequality:
\begin{equation}\label{eqn_decoupling_l4}
    \norm {f}_{L^4(\R^3)}\leq C_\Ge\Gd^{-\Ge}\# \mathcal P^{\frac 1 4}\left(\sum_{P\in \mathcal P}\norm{f_P}_{L^4(\R^3)}^4\right)^{\frac 1 4}.
\end{equation}
If, in addition, the Hessian of $\phi$ is positive-semidefinite on $[-1,1]^2$, i.e. $\phi$ is convex, then \eqref{eqn_decoupling_l4} can be strengthened to the $l^2(L^4)$-decoupling inequality:
\begin{equation}\label{eqn_decoupling_l2}
    \norm {f}_{L^4(\R^3)}\leq C_\Ge\Gd^{-\Ge}\left(\sum_{P\in \mathcal P}\norm{f_P}_{L^4(\R^3)}^2\right)^{\frac 1 2}.
\end{equation}
\end{thm}

{\it Remark.} In this theorem and Theorem \ref{thm_Bourgain_Demeter} below, there is a ``first principles" decoupling
in that $l^2(L^4)$ estimates are stronger than the corresponding $l^4(L^4)$ estimates, which follows from a simple application of the H\"older's inequality. Also, if in some cases the sum on the right hand side has at most $O(\log\Gd^{-1})$ terms, a trivial application of the triangle inequality followed by H\"older's inequality gives rise to a factor of at most $(\log\Gd^{-1})^C$, which is majorised by $C_\Ge\Gd^{-\Ge}$ for any $\Ge>0$. See the argument right after \eqref{eqn_radially_dyadic} for an example.

This principle will be used extensively in this paper, for instance, the dyadic decompositions at the start of Sections \ref{sec_reduction} and \ref{sec_proof_lemma}.

\subsection{Main ingredients of proof}
The following results of Bourgain and Demeter will serve as the cornerstone of our argument.
\begin{thm}[Bourgain-Demeter, \cites{BD2015,BD2017}\label{thm_Bourgain_Demeter}]
Let $\phi:[-1,1]^2$ be a $C^3$-function with nonsingular Hessian. Then for any $\Ge>0$, there is a constant $C_\Ge$, depending on $\phi$ and $\Ge$ only, such that the following is true. For every $0<\Gd<1$, denote by $\mathcal P_\Gd$ a cover of $[-1,1]^2$ by finitely overlapping squares of side length $\Gd^{1/2}$. Then for any function $f:\R^3\to \C$ with Fourier support on the $\Gd$-neighbourhood of the graph of $\phi$ above $[-1,1]^2$, we have
\begin{equation}\label{eqn_BD_decoupling_l4}
    \norm {f}_{L^4(\R^3)}\leq C_\Ge\Gd^{-\Ge}\# \mathcal P_\Gd^{\frac 1 4}\left(\sum_{P\in \mathcal P_\Gd}\norm{f_P}_{L^4(\R^3)}^4\right)^{\frac 1 4}.
\end{equation}
If, in addition, the Hessian of $\phi$ is positive-definite on $[-1,1]^2$, i.e. $\phi$ is convex, then \eqref{eqn_BD_decoupling_l4} can be strengthened to
\begin{equation}\label{eqn_BD_decoupling_l2}
    \norm {f}_{L^4(\R^3)}\leq C_\Ge\Gd^{-\Ge}\left(\sum_{P\in \mathcal P_\Gd}\norm{f_P}_{L^4(\R^3)}^2\right)^{\frac 1 2}.
\end{equation}
\end{thm}
The $l^2$-inequality \eqref{eqn_BD_decoupling_l2} follows from Section 7 of \cite{BD2015} when $n=3$ and $p=4$. The $l^4$-inequality \eqref{eqn_BD_decoupling_l4} follows from Theorem 1.1 of \cite{BD2017} when $n=3$ and $p=4$, combined with a standard Pramanik-Seeger type iteration introduced in \cite{PS2007} which is also the main argument in Section 7 of \cite{BD2015}.

\subsubsection{Main idea of proof}
Starting from Theorem \ref{thm_Bourgain_Demeter}, the major difficulty of Conjecture \ref{conj_BDK2019}
lies in handling the decoupling of the region near $\{\det D^2 \phi=0\}$. The reason we choose mixed-homogeneous polynomials is that the set $\{\det D^2 \phi=0\}$ behaves well: it is a union of algebraic curves of the form $x^s-\lambda y^r =0$, $x=0$ or $y=0$. Mixed-homogeneity also plays a role in reducing the domain of decoupling from $[-1,1]^2$ to an annulus (Section \ref{sec_reduction}). This allows us to separate the curves along which $\det(D^2 \phi)$ vanishes.  For now, let us restrict ourselves to one algebraic curve $x^s-\lambda y^r=0$. 

Let $R$ be a neighborhood of the algebraic curve. Following \cite{BDK2019}, the general procedure to decouple $R$ is as follows. First, we apply dyadic decomposition to localize $\phi$ to the region where $|\det D^2 \phi|$ is essentially $\sigma$, which we denote by $R(\sigma)$. Second, we decouple $R(\sigma)$ into rectangles $R\in\mathcal{R}$ using a preliminary cylindrical decoupling. Third, each $R\in \mathcal{R}$ can be rescaled so that the rescaled phase function has nonsingular Hessian. Finally, we apply Theorem \ref{thm_Bourgain_Demeter} to the rescaled piece to get the desired decoupling inequality.

However, there is a significant obstacle when applying the above procedure in our case. It can be described as follows. The orientations of maximal $\delta$-flat pieces are governed by the eigenvectors of $D^2 \phi$. Unfortunately, these eigenvectors are unstable over $R(\sigma)$. More precisely, in a simplified scenario as in Figure \ref{fig_ite_cyl_dec} below, let $R(\sigma)$ be the long strip and $\mathcal{R}$ be the collection of shaded rectangles. The orientations of shaded rectangles correspond to the eigenvector directions of $D^2\phi$. These orientations are the directions along which we rescale the shaded rectangles as the third step above. Due to the big difference in the orientations, there is no way we can apply cylindrical decoupling to decompose $R(\sigma)$ into the $R\in\mathcal{R}$ directly.

To overcome the obstacle, we will apply an iterative cylindrical decoupling on $R(\sigma)$ instead. We first decouple it to some intermediate rectangles. These are the smallest possible axis parallel rectangles into which we can cylindrically decouple the strip due to the eigenvector issues. On each of these rectangles, we perform a different shear transform so that the eigenvectors are more stable in each intermediate rectangle when compared to the original strip. Then we can decouple them further into rectangles of finer scale thanks to the improved stability. After some iterations, we will arrive at the shaded rectangles. In general, the number of iterations needed depends on the multiplicity of the factor $x^s - \lambda y^r$ in $\det D^2 \phi$. This can be thought of as a cylindrical version of Pramanik-Seeger iteration for functions similar to $x^2 + \sigma \psi(x,y)$.

\begin{figure}[h]
    \centering
        \begin{tikzpicture}
    
    \draw[] (0,0) -- (0,2);
    \foreach \i in {0,0.5}
    {
    \draw[fill=gray!20,fill opacity=0.25] (\i,0) -- (\i+0.5,0) -- (\i+0.5,2) -- (\i,2);
    }
    \node[] at (1.35,1) {\large $\cdot \cdot \cdot$};

    \draw[] (2.7+0.5,0.075) -- (2.7-0.3+0.033+0.5,2);
    \foreach \i in {0+2.7+0.5,0.5+2.7+0.5}
    {
    \draw[fill=gray!20,fill opacity=0.25] (\i,0.075) --(\i+0.015,0) -- (\i+0.5,0.075) -- (\i+0.5-0.3,2.075) -- (\i-0.3+0.033,2);
    }
    \node[] at (1.35+3.1,1) {\large $\cdot \cdot \cdot$};

    \draw[] (5.4+1,0.225) -- (5.4-0.9+0.1+1,2);
    \foreach \i in {0+5.4+1,0.5+5.4+1}
    {
    \draw[fill=gray!20,fill opacity=0.25] (\i,0.225)-- (\i+0.1,0) -- (\i+0.5,0.225) -- (\i+0.5-0.9,2.225) -- (\i-0.9+0.1,2);
    }
    \node[] at (1.35+6.2,1) {\large $\cdot \cdot \cdot$};
    
    \foreach \i in {0,2.7,5.4}
    {
    \draw[, thick] (\i,0) -- (\i+2.7,0) -- (\i+2.7,2) -- (\i,2) -- (\i,0);
    }
    \node[] at (8.7,1) {\huge $\cdot \cdot \cdot$};
    
    \draw (0,0) -- (10,0) -- (10,2) -- (0,2) -- (0,0);
    
    \end{tikzpicture}
    \caption{An iterative cylindrical decoupling}
    \label{fig_ite_cyl_dec}
\end{figure}
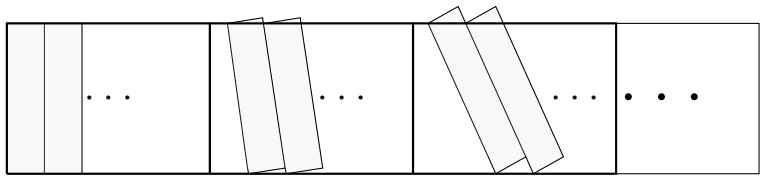

{\it Outline of the paper.} In Section \ref{sec_notation} we introduce the necessary notation. In Section \ref{sec_reduction} we use the properties of mixed-homogeneous polynomials to perform a series of reductions and classify different cases we will encounter. In Sections \ref{sec_axis_1} through \ref{sec_dec_curve_p2} we solve the decoupling problems in each case. In Section \ref{sec_convex} we note the changes needed to improve from $l^4(L^4)$ decoupling inequalities to stronger $l^2(L^4)$ inequalities as indicated in \eqref{eqn_decoupling_l2}. Lastly, Section \ref{sec_appendix} is an appendix, which studies some algebraic properties of mixed-homogeneous polynomials, as well as some fundamental notions in decoupling theory.

{\it Acknowledgement.} The first author would like to thank Betsy Stovall for her advice and help throughout the project. The second author would like to thank his advisor Malabika Pramanik for suggesting a prototype of this problem at an early stage. The authors would also like to thank the anonymous referee for their thoughtful comments. The first author was supported by NSF DMS-1653264.

\section{Notation}\label{sec_notation}
We will introduce the most frequently used terminologies in this paper.
\begin{enumerate}
    \item First of all, instead of writing out decoupling theorems in full detail as in Theorems \ref{thm_Bourgain_Demeter} and \ref{thm_decoupling}, we will
often state their short form as follows. For example, Theorem \ref{thm_decoupling} will be simply abbreviated as
\begin{equation*}
    \text{$[-1,1]^2$ can be decoupled into $(\phi,\Gd)$-flat rectangles.}
\end{equation*}
Here and throughout this paper, we always implicitly assume the rectangles have bounded overlap, that is, there is some constant $C$ depending on $\phi$ only, such that the sum of the indicator functions of these subsets is always bounded above by $C$.

We remark that that this terminology refers to both the $l^4$-decoupling \eqref{eqn_decoupling_l4} and the $l^2$-decoupling \eqref{eqn_decoupling_l2} if the Hessian is positive-semidefinite.

We also remark that each time we write this terminology, we implicitly lose a factor of size at most $C_{\varepsilon}\delta^{-\varepsilon}$. We will see that the number of uses of this terminology depends only on $\phi$. Therefore, the loss is acceptable and thus is omitted in the argument.

\item We will write $(x,y,z)$ to denote the frequency variables. This is not standard, but since we only do explicit computations on the physical space in the proof of Proposition \ref{prop_curve_neig_dec}, this will not cause confusion.

\item For any set $A\sub \R^2$, any smooth function $\phi:A\to \R$ and any $\Gd>0$, denote the (vertical) $\Gd$-neighbourhood of the graph of $\phi$ above $A$ by
$$
\mathcal N^\phi_{\Gd}(A)=\{(x,y,z):(x,y)\in A, |z-\phi(x,y)|<\Gd\}.
$$

\item Let $A$ be a regular geometric figure (such as a parallelogram) and $B$ be any set that lies in the same Euclidean space as $A$. $A,B$ are said to be equivalent if there is an absolute constant $C>1$ (which may depend on $\phi$) such that $C^{-1}A\sub B\sub CA$. Here $CA$ means the dilation of $A$ by a factor of $C$ with respect to the centre of $A$.

\item We use the notation $A=O(B)$, or $A\lesssim B$ to mean there is an absolute constant $C$ such that $A\leq CB$. In this paper, all implicit constants are allowed to depend on $\phi$.
\end{enumerate}

\section{Preliminary reductions}\label{sec_reduction}
In this section we will reduce the proof of Theorem \ref{thm_decoupling} to slightly simpler cases.
\subsection{Radially dyadic decomposition}
We first use the properties of a mixed-homogeneous polynomial to reduce the decoupling problem to the closed ``annulus"
$$
A:=[-2,2]^2\bs (-1,1)^2.
$$
\begin{lem}\label{lem_scaling}
Theorem \ref{thm_decoupling} holds if we can decouple the annulus $A$ into $(\phi,\Gd)$-flat rectangles.
\end{lem}
\begin{proof}
Let $q,r,s$ be positive integers such that $\phi(x,y) = \rho^{-q} \phi(\rho^{r} x, \rho^{s} y)$ for all $\rho >0$.

We decompose $[-1,1]^2$ into $A_{j}$, $j\geq 0$ as follows. Define
$$
A_0=[-\Gd^{r/q},\Gd^{r/q}]\times [-\Gd^{s/q},\Gd^{s/q}],
$$
and for $j\geq 1$,
$$
A_{j}=([-(2^{j}\Gd)^{r/q},(2^{j}\Gd)^{r/q}]\times [-(2^{j}\Gd)^{s/q},(2^{j}\Gd)^{s/q}])\bs A_{j-1}.
$$
Note that we have at most $O(\log \Gd^{-1})$ such $j$ since we are in $[-1,1]^2$. We claim that it suffices to decouple each $A_j$, $j\geq 0$ with respect to $\phi$.

Indeed, suppose for each $j$ we have found a family of boundedly overlapping rectangles $\mathcal P_j$, such that $\phi$ is $\Gd$-flat over each $P\in \mathcal P_j$ and that for any function $f_{A_j}$ with Fourier support on $\mathcal N^\phi_{\Gd}(A_j)$, we have
\begin{equation}\label{eqn_radially_dyadic}
    \norm {f_{A_j}}_{L^4(\R^3)}\leq C_\Ge\Gd^{-\Ge}\# \mathcal P_j^{\frac 1 4}\left(\sum_{P\in \mathcal P_j}\norm{f_P}_{L^4(\R^3)}^4\right)^{\frac 1 4}.
\end{equation}
Then we simply define $\mathcal P:=\cup_j \mathcal P_j$. Given any $f$ Fourier supported on $\mathcal N^\phi_\Gd([-1,1]^2)$, we simply use triangle and H\"older's inequalities:
\begin{align*}
    \norm {f}_{L^4(\R^3)}
    &\leq \sum_j \norm {f_{A_j}}_{L^4(\R^3)}\lesssim \log(\Gd^{-1})^{3/4} \left(\sum_j \norm {f_{A_j}}^4_{L^4(\R^3)}\right)^{1/4}\\
    &\lesssim_\Ge \Gd^{-\Ge} \left(\sum_j C^4_\Ge\Gd^{-4\Ge}\# \mathcal P_j\sum_{P\in \mathcal P_j}\norm{f_P}_{L^4(\R^3)}^4\right)^{1/4}\\
    &\leq C_\Ge\Gd^{-2\Ge}\left( \# \mathcal P\sum_{P\in \mathcal P}\norm{f_P}_{L^4(\R^3)}^4\right)^{1/4},
\end{align*}
as required. The argument for the $l^2$-inequality is similar and easier.

Now we come to the decoupling of each $A_j$, $j\geq 0$. Over $A_0$, taking $\rho=\Gd^{-1/q}$ and using the mixed-homogeneity of $\phi$ we easily have 
$$
\sup_{(x,y)\in A_0}|\phi(x,y)|=\Gd \sup_{(x,y)\in [-1,1]^2}|\phi(x,y)|\lesssim \Gd,
$$
and hence $\phi$ is $\Gd$-flat over $A_0$.

Thus it remains to decouple each $A_j$, $j\geq 1$ with respect to $\phi$. The following simple rescaling reduces the problem to decoupling the annulus $A=[-2,2]^2\bs [-1,1]^2$ into $\Gd$-flat rectangles. To see this, fix a $j\geq 1$ and write $\rho=(2^{j-1} \Gd)^{1/q}\geq \Gd^{1/q}$. Let $f$ have Fourier support on $\mathcal N^\phi_{\Gd}(A_j)$, and define $g$ by the relation
$$
\hat g(x,y,z)=\hat f(\rho^{r} x,\rho^{s} y,\rho^q z).
$$
Thus, for any $(x,y,z)$ in the support of $\hat g$, we have $(x,y)\in [1,2]^2$, and using the mixed-homogeneity of $\phi$ we have $|z-\phi(x,y)|<\rho^{-q}\Gd$. Hence $g$ is Fourier supported on $\mathcal N^\phi_{\rho^{-q}\Gd}(A)$.

Thus, to decouple $A_j$ into $(\phi,\Gd)$-flat rectangles, it is equivalent to decoupling $A$ into $(\phi,\rho^{-q}\Gd)$-flat rectangles (note that $\rho^{-q}\Gd\leq 1$ since $\rho\geq \Gd^{1/q}$). The remaining argument is simple but we include it here for completeness.

Suppose we have already managed to decouple $A$ into $(\phi,\rho^{-q}\Gd)$-flat rectangles. That is, for some $C_\Ge$ depending on $\phi,\Ge$ only, there is a boundedly overlapping family $\mathcal P'$ covering $A$, such that $\phi$ is $\rho^{-q}\Gd$-flat over each $P'\in \mathcal P'$ and
$$
\norm {g}_{L^4(\R^3)}\leq C_\Ge(\rho^{-1}\Gd)^{-\Ge}\# \mathcal P'^{\frac 1 4}\left(\sum_{P'\in \mathcal P'}\norm{g_{P'}}_{L^4(\R^3)}^4\right)^{\frac 1 4}.
$$
For each $P'\in \mathcal P'$, let $P$ be the image of $P'$ under the dilation mapping $(x,y)\mapsto (\rho^r x,\rho^s y)$.\footnote{Strictly, $P$ may not be a rectangle if $s\neq r$; however, this technicality is easily solved since $P$ is always a parallelogram, which can be slightly enlarged to a rectangle that is equivalent to it (see Section \ref{sec_notation}).} Then, by the scaling invariance of the Fourier transform, we can rescale back to $f$ to get
$$
\norm {f}_{L^4(\R^3)}\leq C_\Ge\Gd^{-\Ge}\# \mathcal P^{\frac 1 4}\left(\sum_{P\in \mathcal P}\norm{f_{P}}_{L^4(\R^3)}^4\right)^{\frac 1 4},
$$
where we have used $\rho^\Ge\leq 1$ and $\#\mathcal P=\#\mathcal P'$.

The argument of the $l^2$-inequality is similar and easier, since a positive dilation does not change the sign of the determinant.
\end{proof}

\subsection{Localization}
Consider the Hessian determinant of $\phi$ defined by
\begin{equation}
    K(x,y)=\det(D^2\phi)(x,y)=(\phi_{xx}\phi_{yy}-\phi_{xy}^2)(x,y).
\end{equation}
If $K(x,y)$ has no zero on $A=[-2,2]^2\bs [-1,1]^2$, then we may directly use Theorem \ref{thm_Bourgain_Demeter} to conclude the proof. Otherwise, we will first analyse the zero set of $K$, localize to each component of the zero set, and treat them individually. 

\subsubsection{The case of identically zero Gaussian curvature}
We first settle the case when $K$ vanishes identically. In this case, $\phi:\R^2\to \R$ defines a complete surface of zero Gaussian curvature. Thus, by Section 5.8 of \cite{Docarmo2016}, it must be a cylinder or a plane. In either case, by a rotation we may write $\phi(x,y)=Cx+\psi(y)$ for a univariate polynomial $\psi$. Since the linear term $Cx$ is negligible in decoupling, we may apply lower dimensional decoupling to reduce the problem to decoupling for the univariate polynomial $\psi$, which has been solved in \cites{BGLSX,Demeter2020,Yang2} as we mentioned in the introduction.

From now on we always assume we are in the nontrivial case when $K$ vanishes somewhere on $A$ but is not identically zero.

\subsubsection{Facts about mixed-homogeneous polynomials}
We shall need the following facts about mixed-homogeneous polynomials. First, it is easy to observe that the Hessian determinant of a mixed-homogeneous polynomial satisfying \eqref{eqn_mixed_homo_def0} is also mixed-homogeneous, with the same exponents $r,s$. Using the factorization property \eqref{eqn_factorisation_intro}, we can write $K$ as follows:
\begin{equation*}
   \phi(x,y)=x^{\nu_1}y^{\nu_2}\prod_{j=1}^{M} (x^s - \lambda_j y^r)^{n_j}P(x^s,y^r),
\end{equation*}
for some nonnegative integers $\nu_1,\nu_2,M,n_j$, nonzero real numbers $\Gl_j$, and a homogeneous polynomial $P$ that never vanishes except at the origin; thus $P$ is bounded away from $0$ on $A_0$.

Therefore, for the zero set of $K$ on the annulus $A=[-2,2]^2\bs [-1,1]^2$, we have the following scenarios, which can happen simultaneously.
\begin{enumerate}
    \item If $\nu_1 \neq 0$, i.e. $x \mid K$, then $K$ vanishes on the $y$-axis.
    \item If $\nu_2 \neq 0$, i.e. $y \mid K$, then $K$ vanishes on the $x$-axis.
    \item If $n_j>0$ for some $j$, then $K$ vanishes on the curve $x^s - \lambda_j y^r =0$.
\end{enumerate}

\subsubsection{Separation of zero set of $K$}\label{sec_choose_c}
Note that the zero set of $K$ consists of curves that only intersect at the origin. Hence, within the annulus $A$, the zero set of $K$ has finitely many positively separated connected components, each of which is either part of the $x$- or $y$-axis, or part of the curve $x^s - \lambda_j y^r =0$ with nonzero curvature.

Let $c_\phi>0$ be a suitably small constant depending on $\phi$ only. We will not explicitly define its value here, but for a couple of times later in the paper we will impose several conditions that $c_\phi$ has to satisfy.

Now we perform a preliminary decomposition of $A$ into finitely many disjoint regions: the regions where $(x,y)$ is within distance $c_\phi$ of the zero set of $K$ and elsewhere. By triangle and H\"older's inequalities it suffices to consider each region. On the region away from the zero set of $K$, we use Theorem \ref{thm_Bourgain_Demeter} to conclude the proof. Thus, we are left with the following tasks:
\begin{enumerate}
    \descitem {(A)} Decoupling the $c_\phi$-neighbourhood of the $x$-and $y$-axes within the annulus $A$ into $(\phi,\Gd)$-flat rectangles.
    \descitem {(B)} Decoupling the $c_\phi$-neighbourhood of the curve $x^s - \lambda_j y^r =0$ within the annulus $A$ into $(\phi,\Gd)$-flat rectangles.
\end{enumerate}

\subsection{Classification of scenarios}\label{sec_classification}
In this section, for both tasks mentioned in Section \ref{sec_choose_c} above, we classify the mixed-homogeneous polynomials, so that we can treat each scenario separately in the following sections.

\subsubsection{Decoupling neighbourhoods of axes}
To complete Task \descref{(A)}, we first note that by a rotation it suffices to consider the $c_\phi$-neighbourhood of the $y$-axis. For this purpose, we consider the following cases:
\begin{enumerate}
    \descitem {(A1)} $\phi$ has a factor $y^2$.
    \descitem {(A2)} $\phi$ cannot be divided by $y^2$.
\end{enumerate}
As we shall see, Case \descref{(A1)} is much simpler and will be studied in Section \ref{sec_axis_1}. The more difficult Case \descref{(A2)} will be studied in Section \ref{sec_dec_axi_p2}.

\subsubsection{Decoupling neighbourhoods of curves}
To complete Task \descref{(B)}, we have the following cases:
\begin{enumerate}
    \descitem {(B1)} $\phi$ has a factor $(x^s-\Gl_j y^r)^2$.
    \descitem {(B2)} $\phi$ cannot be divided by $(x^s-\Gl_j y^r)^2$.
\end{enumerate}
As we shall see, Case \descref{(B1)} resembles Case \descref{(A1)}, and Case \descref{(B2)} resembles Case \descref{(A2)}. We will solve Case \descref{(B1)} in Section \ref{sec_dec_cur_p1} and Case \descref{(B2)} in Section \ref{sec_dec_curve_p2}.

\section{Decoupling axis neighbourhoods: Part I}\label{sec_axis_1}

In this section we deal with Case \descref{(A1)} introduced in Section \ref{sec_classification}, namely, when $y^2$ divides $\phi$. Without loss of generality, we only consider the rectangle $[1,2]\times [0,c_\phi]$ in the first quadrant.

\begin{prop}
Let $\phi(x,y)$ be a mixed-homogeneous polynomial with a factor $y^2$. Then $[1,2]\times [0,c_\phi]$ can be decoupled into $(\phi,\Gd)$-flat rectangles that are also axis-parallel.
\end{prop}

\begin{proof}
Write $k\geq 2$ as the largest positive integer such that $y^k$ divides $\phi$. Then we write $\phi(x,y)=y^k P$. If $P\equiv C\neq 0$, then we just need to do a lower dimensional decoupling for the curve $y\mapsto Cy^k$ over $[0,c_\phi]$. Using Theorem 1.4 of \cite{Yang2} with $p=4$, we can thus decouple $[1,2]\times [0,c_\phi]$ into 
$$
\{[1,2]\times [(j-1)^{2/k}\Gd^{1/k},j^{2/k}\Gd^{1/k}],1\leq j\lesssim \Gd^{-1/2}\},
$$
on each rectangle of which $\phi$ is $\Gd$-flat.

If $P\not\equiv C$, then by direct computation, we have
\begin{align*}
    K(x,y)&=\det(D^2\phi)(x,y)\\
    &=y^{2k-2}[(k(k-1)P+2kyP_y+y^2P_{yy})P_{xx}-(kP_x+yP_{xy})^2]\\
    &:=y^{2k-2}S(x,y).
\end{align*}
Thus $y^{2k-2}$ divides $K(x,y)$. In view of the mixed-homogeneity, using \eqref{eqn_expansion} we may write $P=x^d+yR(x,y)$, $d\geq 1$. By direct computation 
$$
S(x,0)=k(k-1)d(d-1)x^{2d-2}-k^2d^2x^{2d-2},
$$
which is never zero for $x\in [1,2]$.

Now we perform a dyadic decomposition: let $R_0=[1,2]\times [0,\Gd^{1/k}]$ and for $j\geq 1$, define
$$
R_j=[1,2]\times [2^{j-1}\Gd^{1/k},2^j \Gd^{1/k}].
$$
It suffices to decouple each $R_j$ due to our tolerance of logarithmic losses. On $R_0$ we have $|\phi(x,y)|\lesssim\Gd$, and hence $R_0$ is $\Gd$-flat. 

For $j\geq 1$, we denote $\Gs=2^{j-1}\Gd^{1/k}\in [\Gd^{1/k},c_\phi]$. We are going to decouple $R_j=[1,2]\times [\Gs,2\Gs]$ into $(\phi,\Gd)$-flat rectangles that are axis parallel.

The following argument is again rescaling. Let $f$ have Fourier support on $\mathcal N^\phi_\Gd(R_j)$. Define $g$ via the relation
$$
\hat g(x,y,z)=\hat f(x,\Gs y,\Gs^k z),
$$
and thus $g$ has Fourier support on $\mathcal N^{\psi}_{\Gs^{-k}\Gd}([1,2]\times [1,2])$ where
$$
\psi(x,y):=\Gs^{-k}\phi(x,\Gs y)=y^k P(x,\Gs y).
$$
Note that $\psi$ has bounded coefficients. By the previous computation,
$$
\det(D^2 \psi)(x,y)=y^{2k-2}S(x,\Gs y).
$$
Since $S(x,0)$ is never zero on $[1,2]$, by continuity, if $c_\phi$ in Section \ref{sec_choose_c} is chosen to be small enough, then $|S(x,y)|$ is bounded below for all $(x,y)\in [1,2]\times [0,c_\phi]$. As a result, we have $|\det(D^2 \psi)(x,y)|\gtrsim 1$ over $[1,2]^2$. Now we may apply Theorem \ref{thm_Bourgain_Demeter} to decouple $[1,2]^2$ into $(\psi,\Gs^{-k}\Gd)$-flat squares with side length $\Gs^{-k/2}\Gd^{1/2}$ that are also axis-parallel. Scaling back as in the end of proof of Lemma \ref{lem_scaling}, we thus have decoupled $R_j=[1,2]\times [\Gs,2\Gs]$ into $(\phi,\Gd)$-flat rectangles that are axis-parallel and have dimension
$$
\Gs^{-k/2}\Gd^{1/2}\times \Gs^{1-k/2}\Gd^{1/2}
$$
in the $x$-and $y$-directions, respectively.

\end{proof}

\section{Decoupling axis neighbourhoods: Part II}\label{sec_dec_axi_p2}
In this section we complete Task \descref{(A)} when we are in Case \descref{(A2)} introduced in Section \ref{sec_classification}. We again only consider the part in the first quadrant.

By Proposition \ref{prop_y2_not_divide} in the appendix, if $y^2$ does not divide $\phi$, then $\phi$ must be of the form
\begin{equation}\label{eqn_part2}
\phi(x,y)=Cx^m+ yP(x,y),
\end{equation}
for some $C\neq 0$, $m\geq 2$ and some polynomial $P$.

\begin{prop}\label{prop_part2}
Let $\phi(x,y)$ be a mixed-homogeneous polynomial not divisible by $y^2$. Then $[1,2]\times [0,c_\phi]$ can be decoupled into $(\phi,\Gd)$-flat rectangles.
\end{prop}
The main ingredient of the proof is an iterative decoupling lemma, namely, Lemma \ref{lem_cylinderial_dec}, which we will prove by induction on scales in Section \ref{sec_proof_lemma}. Before stating and proving the lemma, we return to some preliminary reductions that are used to prove Proposition \ref{prop_part2}.

We remark the reader that it is advisable to study the following two model cases. The first and easier one is
$$
\phi(x,y)=x^4+6x^2 y+6y^2,
$$
where $y$ but not $y^2$ divides $\det(D^2\phi)$. The second and harder one is
$$
\phi(x,y)=8 x^8 + 112 x^6 y + 504 x^4 y^2 + 756 x^2 y^3 + 189 y^4,
$$
where $y^5$ but not $y^6$ divides $\det(D^2\phi)$.

\subsection{The preparation stage}

\subsubsection{Dyadic decomposition}\label{sec_dyadic_part2}
Let $k\geq 0$ be the largest integer such that $y^k$ divides $\det(D^2 \phi)$. If $k=0$ then we are done, since then $\det(D^2\phi)$ is bounded below on $[1,2]\times [0,c_\phi]$ and we can apply Theorem \ref{thm_Bourgain_Demeter}. So we assume $k\geq 1$.

Dyadically decompose $[1,2]\times [0,c_\phi]$ as follows. Let $R_0=[1,2]\times [0,\Gd^{\frac 1 {k+2}}]$, and for $j\geq 1$, define
$$
R_j=[1,2]\times [2^{j-1}\Gd^{\frac 1 {k+2}},2^j \Gd^{\frac 1{k+2}}].
$$
It suffices to decouple each $R_j$ since we can afford logarithmic losses. The choice of the power of $\Gd$ will be clear later.

We will first decouple $R_j$, $j\geq 1$. The technique introduced here applies to $R_0$ in a similar and easier way.

Denote $\Gs=2^{j-1}\Gd^{\frac 1 {k+2}}\in [\Gd^{\frac 1 {k+2}},c_\phi]$, and so $R_j=[1,2]\times [\Gs,2\Gs]$.

\subsubsection{Rescaling}
Define
\begin{equation}\label{eqn_tilde_phi}
   \tilde \phi(x,y):=\phi(x,\Gs y+\Gs)-\text{linear terms}, 
\end{equation}
and so it suffices to decouple $[1,2]\times [0,1]$ into $(\tilde \phi,\Gd)$-flat rectangles. Note that each coefficient of $\tilde \phi$ is a polynomial in $\Gs$ with uniformly bounded coefficients.

We have an important observation here. Since $k$ is the largest integer such that $y^k$ divides $\det(D^2\phi)$, if $c_\phi$ defined in Section \ref{sec_choose_c} is chosen to be small enough, then using \eqref{eqn_part2}, on $[\frac 1 2,\frac 5 2]\times [\Gs,2\Gs]$ we have
\begin{equation}\label{eqn_estimates_phi}
|\phi_{xx}|\sim 1,\quad |\det(D^2\phi)|\sim \Gs^{k},\quad |\partial_x\det(D^2\phi)|=O( \Gs^{k}).
\end{equation}
Combined with \eqref{eqn_tilde_phi}, \eqref{eqn_estimates_phi} implies that on $[\frac 1 2,\frac 5 2]\times [0,1]$ we have the relations
\begin{equation*}
    |\tilde\phi_{xx}|\sim 1,  \quad |\tilde \phi_{xy}|\lesssim \Gs,\quad |\tilde \phi_{xxy}|\lesssim \Gs,
\end{equation*}
and
\begin{equation*}
    |\det D^2 \tilde\phi(x,y)| \sim \sigma^{k+2},\quad |\partial_x (\det D^2 \tilde\phi)(x,y)| \lesssim \sigma^{k+2}.
\end{equation*}

\subsection{The family $\mathcal A_l$}\label{sec_A_j}
Recall our task now is to decouple $[0,1]^2$ into $(\tilde \phi,\Gd$)-flat rectangles. The main idea for the proof is by induction, which requires a series of technical preparations.

For future purposes, let us also introduce the following family $\mathcal A_l$ of polynomials which depends only on $\phi,c_\phi,\Gs$.

\begin{defn} \label{def_A_l}
Let $0\leq \Gs\leq c_\phi$. For $l\geq 0$, we define $\mathcal A_l=\mathcal A_l(\Gs)$ to be the collection of polynomials $\tau(x,y)$ with the following properties: $\tau$ contains no linear term, all coefficients of $\tau$ are polynomials in $\Gs^{1/2}$, and $\tau$ satisfies the following derivative estimates over $[\frac 1 2,\frac 5 2]\times [0,1]$:
    \begin{equation}\label{eqn_derivatives}
    |\tau_{xx}|\sim 1,\quad | \tau_{xy}|\lesssim \Gs,\quad | \tau_{xxy}|\lesssim \Gs,
\end{equation}
and
\begin{equation}\label{eqn_partial_det}
    |\det D^2 \tau(x,y)| \lesssim \sigma^{l},\quad |\partial_x (\det D^2 \tau)(x,y)| \lesssim \sigma^{l}.
\end{equation}
where the bounds of coefficients of $\tau$ and the implicit constants depend on $\phi$ only.
\end{defn}

Note that $\tilde \phi$ belongs to $\mathcal A_{k+2}$ by definition. In the first inequality of \eqref{eqn_partial_det} we choose $\lesssim$ instead of $\sim$ since we also want to treat the decoupling of $R_0$ where $\det(D^2\phi)$ is not bounded away from zero.

\subsection{Cylindrical decoupling}
We are now ready to state the main lemma for decoupling.

\begin{lem}[Cylindrical decoupling]\label{lem_cylinderial_dec}
Let $\tau$ be a polynomial in the family $\mathcal A_l$. Then the rectangle $[1,2]\times [0,1]$ can be decoupled into $(\tau,\Gs^l)$-flat rectangles $R$. Moreover, for each such $R$, if we let $L_1,L_2$ be its two adjacent sides, then they obey the following geometric properties (see Figure \ref{fig_lem_cylinderial_dec}):
\begin{enumerate}
    \item $|L_1|\sim\Gs^{l/2}$, $|L_2|\sim 1$.
    \item The slope of $L_1$ is of the order $O(1)$. 
\end{enumerate} 
\end{lem}

{\it Remark.} By the geometric properties given in the lemma, the rectangles $R$ covering $[1,2]\times [0,1]$ have bounded overlap, and the proportion that each $R$ exceeds $[1,2]\times [0,1]$ is of the order $O(\Gs^{l/2})$. This ensures that the covering rectangles of different dyadic layers $R_j$ will have bounded overlap.

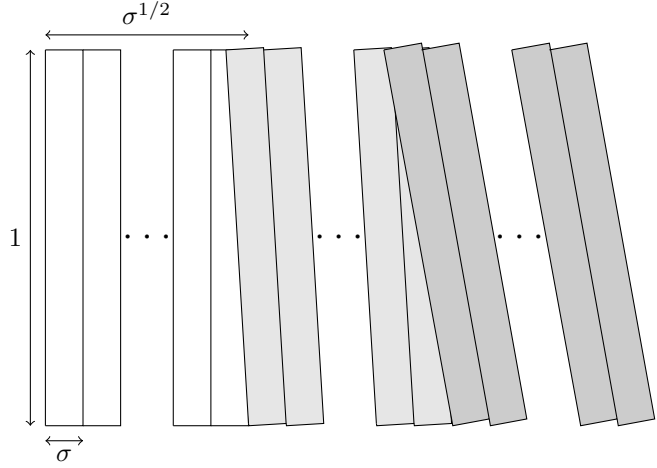
\begin{figure}[h]
    \centering    \begin{tikzpicture}
    
    \draw (0,0) -- (0,5);
    \draw (1.7,0) -- (1.7,5);
    \foreach \i in {0,0.5,1.7,2.2}
    {
    \draw (\i,0) -- (\i+0.5,0) -- (\i+0.5,5) -- (\i,5);
    }
    \node at (1.35,2.5) {\large $\cdot \cdot \cdot$};
    
    \foreach \i in {0+2.7,0.5+2.7,1.7+2.7,2.2+2.7}
    {
    \fill[gray!20] (\i,0.03) --(\i+0.005,0) -- (\i+0.5,0.03) -- (\i+0.5-0.3,5.03) -- (\i-0.3,5);
    }
    \draw[] (2.7,0.03) -- (2.7-0.3,5);
    \draw[] (1.7+2.7,0.03) -- (1.7+2.7-0.3,5);
    \foreach \i in {0+2.7,0.5+2.7,1.7+2.7,2.2+2.7}
    {
    \draw[] (\i,0.03) --(\i+0.005,0) -- (\i+0.5,0.03) -- (\i+0.5-0.3,5.03) -- (\i-0.3,5);
    }
    \node[] at (1.35+2.55,2.5) {\large $\cdot \cdot \cdot$};
    

    \foreach \i in {0+5.4,0.5+5.4,1.7+5.4,2.2+5.4}
    {
    \fill[gray!40] (\i,0.09)-- (\i+0.016,0) -- (\i+0.5,0.09) -- (\i+0.5-0.9,5.09) -- (\i-0.9,5);
    }
    \draw[] (5.4,0.09) -- (5.4-0.9,5);
    \draw[] (1.7+5.4,0.09) -- (1.7+5.4-0.9,5);
    \foreach \i in {0+5.4,0.5+5.4,1.7+5.4,2.2+5.4}
    {
    \draw[] (\i,0.09)-- (\i+0.016,0) -- (\i+0.5,0.09) -- (\i+0.5-0.9,5.09) -- (\i-0.9,5);
    }
    \node[] at (1.35+4.95,2.5) {\large $\cdot \cdot \cdot$};
    
    \draw[<->] (0,-0.2)--(0.5,-0.2);
    \node at (0.25,-0.4) {\footnotesize $\sigma$};
    \draw[<->] (-0.2,0)--(-0.2,5);
    \node at (-0.4,2.5) {\footnotesize $1$};
    \draw[<->] (0,5.2)--(2.7,5.2);
    \node at (1.35,5.5) {\footnotesize $\sigma^{1/2}$};

    \end{tikzpicture}
    \caption{An illustration when $l=2$}
    \label{fig_lem_cylinderial_dec}
\end{figure}

The next proposition can be used to deal with decoupling of $R_j$, $j\geq 1$.
\begin{prop}\label{prop_R_j}
Let $\tau\in \mathcal A_l$ and assume in addition that $|\det(D^2\tau)(x,y)|\sim \Gs^l$. Then for $\Gd\leq \Gs^{l}$, we may further decouple $[1,2]\times [0,1]$ into $(\tau,\Gd)$-flat rectangles of dimensions
$$
\Gd^{1 /2}\times \Gs^{-l/2}\Gd^{ 1 /2},
$$
with their shorter sides having bounded slope.
\end{prop}
We assume Lemma \ref{lem_cylinderial_dec} for now; it will be prove in Section \ref{sec_proof_lemma}.

\begin{proof}[Proof of Proposition \ref{prop_R_j}]
Apply Lemma \ref{lem_cylinderial_dec} to get $(\tau,\Gs^l)$-flat rectangles $R$ with the prescribed geometric properties.

Now fix such an $R$. We need to decouple $R$ further. Let $\rho$ be the rotation (plus a translation) that maps $R$ to an axis-parallel rectangle with the shorter side having zero slope and its lower-left corner being the origin. Define
$$
\psi(x,y):=\Gs^{-l}(\tau\circ \rho^{-1})(\Gs^{l/2}x,y)-\text{linear terms}.
$$
Then by the definition of $\mathcal A_l$, all coefficients of $\psi$ are bounded. Moreover, since $|\det(D^2 \tau)|\sim \Gs^{l}$ over $[1,2]\times [0,1]$, we have $|\det(D^2 \psi)|\sim 1$ over $[0,1]\times [0,1]$.


Hence, we may apply Theorem \ref{thm_Bourgain_Demeter} to decouple $[0,1]\times [0,1]$ into axis-parallel squares of side length $\Gs^{-l/2}\Gd^{1/2}$, on each of which $\psi$ is $\Gs^{-l}\Gd$-flat. Reversing the linear transformations, we have thus decoupled each $R$ into $(\tau,\Gd)$-flat rectangles of dimensions $\Gd^{1/2}\times \Gs^{-l/2}\Gd^{1/2}$, each with the shorter side parallel to the shorter side of $R$. Combining all rectangles $R$, we have thus decoupled $[1,2]\times [0,1]$ into $(\tau,\Gd)$-flat rectangles, and each of them has dimensions $\Gd^{1 /2}\times \Gs^{-l /2}\Gd^{ 1 /2}$, with their shorter sides having bounded slope.

\end{proof}

\subsubsection{Proof of Theorem \ref{prop_part2}}
We may now finally prove Theorem \ref{prop_part2}, while still assuming Lemma \ref{lem_cylinderial_dec}.

\begin{proof}
Recall from the last part of Section \ref{sec_dyadic_part2} that it suffices to decouple each $R_j$, $j\geq 0$.

For $j\geq 1$, we just apply Proposition \ref{prop_R_j} to $\tilde \phi\in \mathcal A_{k+2}$ to decouple $[1,2]\times [0,1]$ into $(\tilde \phi,\Gd)$-flat rectangles of dimensions $\Gd^{1 /2}\times \Gs^{-(k+2)/2}\Gd^{ 1 /2}$, with their shorter sides having bounded slope. Reversing the rescaling, we have thus decoupled $R_j$ into $(\phi,\Gd)$-flat rectangles of dimensions $$
\Gd^{1 /2}\times \Gs^{-k/2}\Gd^{ 1 /2},
$$
with their shorter sides having bounded slope.

It remains to decouple $R_0$. To this end, denote $\Gs:=\Gd^{\frac 1 {k+2}}$. If we let
$$
\overline \phi(x,y)=\phi(x,\Gs y),
$$
then similarly we can check that $\overline \phi$ lies in $\mathcal A_{k+2}$. By Lemma \ref{lem_cylinderial_dec}, we can decouple $[1,2]\times [0,1]$ into $(\overline \phi,\Gs^{k+2})$-flat rectangles $R$ with the given geometric properties. Reversing from $\overline \phi$ to $\phi$, we thus have decoupled $R_0$ into 
$(\phi,\Gs^{k+2})$-flat rectangles $R$. But since $\Gs=\Gd^{\frac 1 {k+2}}$, we have $\Gs^{k+2}=\Gd$
and thus each $R$ is already $(\phi,\Gd)$-flat, as required. Note that each $R$ has dimensions 
$$
\Gd^{ 1/ 2}\times \Gd^{1 /(k+2)},
$$
with their shorter sides having bounded slope.

\end{proof}

Therefore, all that remains is the proof of Lemma \ref{lem_cylinderial_dec}, for which we need some other lemmas below.
\subsection{Coefficient analysis}\label{sec_shear_lemma}
If we compute the eigenvectors of $\det(D^2\tilde\phi)(x,y)$ we may find that they are not necessarily axis-parallel; moreover, they keep changing as $x$ and $y$ vary. This suggests that we introduce a family of suitable shear transformations to reduce to the axis-parallel case. The technical lemma below is the key to the proof.

\begin{lem}\label{lem_shear}
Let $\tau$ be a polynomial in $\mathcal A_l$. Then there is a constant $\mu=\mu(\Gs)$ which is a polynomial in $\Gs$, such that if we define
\begin{equation}\label{eqn_defn_psi}
    \psi(x,y) =\tau(x-\mu y,y),
\end{equation}
then on $[\frac 3 4,\frac 9 4]\times [0,1]$ we have 
\begin{equation}\label{eqn_bound_psi_xx}
   |\psi_{xx}|\sim 1, \quad |\psi_{xy}|\lesssim \Gs, \quad |\psi_{xxy}|\lesssim \Gs,
\end{equation}
and
\begin{equation}\label{eqn_bound_psi_xy}
     \sup_{0\leq y_0\leq 1}|\psi_{xy}(0,y_0)|\lesssim \sigma^{l}.
\end{equation}
The upper bounds of coefficients of $\mu$ and all implicit constants here depends on $\phi$ only.
\end{lem}

\begin{proof}
Let 
$$
\tilde\mu=\frac{\tau_{xy}}{\tau_{xx}}(0,0),
$$
which is a rational function in $\Gs^{1/2}$ of $O(\Gs)$ by assumption. Let $\mu$ be the $(2l-1)$-th order Taylor polynomial of the function $\Gs^{1/2}\mapsto\tilde\mu$ at $\Gs^{1/2}=0$, which has bounded coefficients by the first relation of \eqref{eqn_derivatives}. Thus $|\mu-\tilde\mu|=O(\Gs^{l})$. For $c_\phi$ small enough we thus have the first relation of \eqref{eqn_bound_psi_xx}.

For the second relation, we compute
$$
\psi_{xy}(x,y)
    =-\mu \tau_{xx}(x-\mu y,y)+\tau_{xy}(x-\mu y,y).
$$
Using \eqref{eqn_derivatives} and $\mu=O(\Gs)$, we thus have the bound $\psi_{xy}=O(\Gs)$. The third relation in \eqref{eqn_bound_psi_xx} follows similarly.

To prove \eqref{eqn_bound_psi_xy}, let $y_0\in [0,1]$. We compute
\begin{align*}
    \psi_{xy}(0,y_0)
    &=-\mu \tau_{xx}(-\mu y_0,y_0)+\tau_{xy}(-\mu y_0,y_0)\\
    &=\tau_{xx}(-\mu y_0,y_0)\left(-\mu+\frac {\tau_{xy}}{\tau_{xx}}(-\mu y_0,y_0)\right).
\end{align*}
Since $\tau_{xx}\sim 1$ and $|\tilde\mu-\mu|=O(\Gs^{l})$, to show $|\psi_{xy}(0,y_0)|=O(\Gs^{l})$ it suffices to show that
\begin{equation}\label{eqn_sigma_to_the_l}
   \left|\frac {\tau_{xy}}{\tau_{xx}}(-\mu y_0,y_0)-\frac{\tau_{xy}}{\tau_{xx}}(0,0)\right|=O(\Gs^{l}). 
\end{equation}
For $t\in [0,1]$, define
$$
x(t)=-t\mu y_0,\quad y(t)=ty_0,
$$
and
\begin{equation}\label{eqn_def_nu}
    \nu(t)=\frac {\tau_{xy}}{\tau_{xx}}(x(t),y(t))=\frac {\tau_{xy}}{\tau_{xx}}(-t\mu y_0,ty_0),
\end{equation}
so that 
$$
\nu(1)=\frac {\tau_{xy}}{\tau_{xx}}(-\mu y_0,y_0),\quad \nu(0)=\frac{\tau_{xy}}{\tau_{xx}}(0,0)=\tilde \mu
$$
and \eqref{eqn_sigma_to_the_l} becomes $|\nu(1)-\nu(0)|=O(\Gs^{l})$. By the mean value theorem it then suffices to show that
$\sup_{t\in [0,1]} |\nu'(t)|=O(\Gs^{l})$.

To this end we compute
\begin{align*}
    \nu'(t)
    &=-\mu y_0\partial_x\left(\frac{\tau_{xy}}{\tau_{xx}}\right)(x(t),y(t))+y_0\partial_y\left(\frac{\tau_{xy}}{\tau_{xx}}\right)(x(t),y(t))\\
    &=O(\Gs^{l})-\tilde\mu y_0\partial_x\left(\frac{\tau_{xy}}{\tau_{xx}}\right)(x(t),y(t))+y_0\partial_y\left(\frac{\tau_{xy}}{\tau_{xx}}\right)(x(t),y(t))\\
    &=O(\Gs^{l})+\left[-\nu(t)y_0\partial_x\left(\frac{\tau_{xy}}{\tau_{xx}}\right)(x(t),y(t))+y_0\partial_y\left(\frac{\tau_{xy}}{\tau_{xx}}\right)(x(t),y(t))\right]\\
    &+\left[(\nu(t)-\tilde \mu)y_0\partial_x\left(\frac{\tau_{xy}}{\tau_{xx}}\right)(x(t),y(t))\right]\\
    &:=O(\Gs^{l})+I(t)+II(t).
\end{align*}
Thus
$$
\sup_{t\in [0,1]}|\nu'(t)|\leq O(\Gs^{l})+\sup_{t\in [0,1]}|I(t)|+\sup_{t\in [0,1]}|II(t)|.
$$
For $II$, using \eqref{eqn_derivatives}, we have $\partial_x(\tau_{xy}/\tau_{xx})=O(\Gs)$. Since $0\leq y_0\leq 1$ and $\tilde \mu=\nu(0)$, using mean value theorem we thus have
$$
|II(t)|=|\nu(t)-\nu(0)|\lesssim \Gs\sup_{t\in [0,1]}|\nu'(t)|.
$$
Since $\Gs\leq c_\phi$, if $c_\phi$ in Section \ref{sec_choose_c} is chosen small enough, we have $|II(t)|\leq \sup_{t\in [0,1]}|\nu'(t)|/2$. Thus we have
$$
\sup_{t\in [0,1]}|\nu'(t)|\leq O(\Gs^{l})+\sup_{t\in [0,1]}|I(t)|,
$$
and it suffices to show $|I(t)|=O(\Gs^{k+1})$. But by \eqref{eqn_def_nu} and direct computation,
\begin{align*}
    I(t)
    &=y_0\left[-\frac {\tau_{xy}}{\tau_{xx}}\partial_x\left(\frac{\tau_{xy}}{\tau_{xx}}\right)+\partial_y\left(\frac{\tau_{xy}}{\tau_{xx}}\right)\right](x(t),y(t))\\
    &=y_0\left[-\frac {\tau_{xy}}{\tau_{xx}}\frac {\tau_{xxy}\tau_{xx}-\tau_{xy}\tau_{xxx}}{\tau_{xx}^2}+\frac {\tau_{xyy}\tau_{xx}-\tau_{xy}\tau_{xxy}}{\tau_{xx}^2}\right](x(t),y(t))\\
    &=y_0\left[\frac {\tau_{xy}^2\tau_{xxx}+\tau_{xx}^2\tau_{xyy}-2\tau_{xx}\tau_{xy}\tau_{xxy}}{\tau_{xx}^3}\right](x(t),y(t))\\
    &=y_0\left[\frac {\tau_{xx}\tau_{yy}\tau_{xxx}+\tau_{xx}^2\tau_{xyy}-2\tau_{xx}\tau_{xy}\tau_{xxy}}{\tau_{xx}^3}+\frac {\tau_{xy}^2\tau_{xxx}-\tau_{xx}\tau_{yy}\tau_{xxx}}{\tau_{xx}^3}\right](x(t),y(t))\\
    &=y_0\left[\frac{\partial_x (\det D^2\tau)}{\tau_{xx}^2}-\frac{\tau_{xxx}\det D^2\tau}{\tau_{xx}^3}\right](x(t),y(t)).
\end{align*}
Using the assumptions in \eqref{eqn_derivatives} and \eqref{eqn_partial_det}, we thus have $|I(t)|=O(\Gs^{l})$, as required.
\end{proof}

The consequence of Lemma \ref{lem_shear} is as follows.
\begin{lem}\label{lem_expand}
Let $\tau\in \mathcal A_l$, and let $\psi$ be obtained from $\tau$ as in Lemma \ref{lem_shear}. Then with all linear terms removed, $\psi$ takes the form 
\begin{equation}\label{eqn_psi_expand}
    \psi(x, y)= A(x)+\Gs^{l} (B(y)+xC(y)) + \Gs x^2y D(x,y),
\end{equation}
where $A,B,C,D$ are polynomials with bounded coefficients, and $|A''(x)|\sim 1$ for $x\in [\frac 3 4,\frac 9 4]$.
\end{lem}

\begin{proof}
There are unique polynomials $a,b,c$ such that $\psi$ takes the form
$$
\psi(x,y)=a(x)+xb(y)+x^2 c(x,y).
$$
Using \eqref{eqn_bound_psi_xy} which we just established, we have $c'(y)=O(\Gs^{l})$ for all $0\leq y\le 1$. Since $\psi$ does not have linear terms, all coefficients of $c(y)$ are of the order $O(\Gs^{l})$.

Since $\psi$ is obtained from $\tau$ by composition with a shear transform, for all $0\leq y\leq 1$ we have
$$
\Gs^l\gtrsim |\det(D^2\psi)(0,y)|=|\psi_{xx}(0,y)\psi_{yy}(0,y)-\psi_{xy}(0,y)^2|.
$$
Using the second relation of \eqref{eqn_bound_psi_xx} this forces to $\sup_{0\leq y\leq 1}|\psi_{yy}(0,y)|\lesssim \Gs^l$, which implies that all coefficients of $b''$ are of the order $O(\Gs^{l})$. Since $\psi$ does not have a linear term, all coefficients of $b( y)$ are of the order $O(\Gs^{l})$.

Next, we may write $c(x,y)=d(x)+ye(x,y)$. Using \eqref{eqn_bound_psi_xx} on $y=0$ shows that $|d''(x)|\sim 1$ for $x\in [\frac 3 4,\frac 9 4]$. Also, by the third relation of \eqref{eqn_derivatives}, we see that all coefficients of $e(x,y)$ are of the order $O(\Gs)$. 

Therefore, renaming the terms we can rewrite $\psi$ in the form of \eqref{eqn_psi_expand}.
\end{proof}

\subsection{Proof of Lemma \ref{lem_cylinderial_dec}}\label{sec_proof_lemma}
Now we are ready to prove the main Lemma \ref{lem_cylinderial_dec}.

\begin{proof}
The proof is by induction on $l$. If $l=0$ then we have nothing to do, since then the whole $[1,2]\times [0,1]$ is $(\tau,1)$-flat.

Now assume the lemma holds for some $l-1\geq 0$, and we aim to prove it for $l$.

Start with a polynomial $\tau\in \mathcal A_{l}$ obeying all the assumptions.

Let $M=M_\phi\geq 100$ be an upper bound for all $|\mu|$ arising from Lemma \ref{lem_shear}. To decouple $[1,2]\times [0,1]$, it suffices to decouple each sub-rectangle of the form $[1,2]\times [b,b+(10M)^{-1}]$. We will only prove the case $b=0$ since the other cases follow from a simple translation.

\subsubsection{Shear and lower dimensional decoupling}
Apply Lemma \ref{lem_expand} to get the function $\psi$ obeying \eqref{eqn_psi_expand}.

Then we apply a lower dimensional decoupling at the scale $\Gs$. By Lemma \ref{lem_expand} we have $\psi(x,y)-A(x)=O(\Gs)$ where $A''(x)\sim 1$ for $x\in [\frac 3 4,\frac 9 4]$, and hence we may apply lower dimensional decoupling to $\psi$ to decouple $[\frac 3 4,\frac 9 4]\times [0,(10M)^{-1}]$ into rectangles of the form 
$$
[b',b'+\Gs^{1/2}]\times [0,(10M)^{-1}].
$$
We will further decouple each such rectangle into $(\psi,\Gs^{l})$-flat rectangles.

\subsubsection{Translation and rescaling}
Define
\begin{equation}\label{eqn_defn_eta}
    \eta(x,y):=\Gs^{-1} \psi(\Gs^{1/2}(x-1)+b',y)-\text{linear terms}.
\end{equation}
We now show that $\eta$ lies in $\mathcal A_{l-1}$. Using \eqref{eqn_psi_expand}, we see $\eta$ has uniformly bounded coefficients. Also, by \eqref{eqn_bound_psi_xx} and \eqref{eqn_psi_expand}, we have \eqref{eqn_derivatives} for all $(x,y)\in [\frac 1 2,\frac 5 2]\times [0,(10M)^{-1}]$ for $c_\phi$ small enough. Since translations and shear transforms have determinant $1$, it is direct to see that every relation in \eqref{eqn_partial_det} also holds.

\subsubsection{Applying induction hypothesis}
Apply the induction hypothesis to $\eta$ to decouple the rectangle $[1,2]\times [0,(10M)^{-1}]$ into $(\eta,\Gs^{l-1})$-flat rectangles with dimensions $\Gs^{(l-1)/2}\times 1$,
and with the slope of the shorter sides of the order $O(1)$.

Now reversing the rescaling and combining the rectangles we obtained for each $b'$, we have thus decoupled $[\frac 3 4,\frac 9 4]\times [0,(10M)^{-1}]$ into
$(\psi,\Gs^{l})$-flat rectangles, each with dimensions $\sim\Gs^{l/2}\times 1$ and with the shorter sides having bounded slope. Reversing the shear transformation and using $|\mu|\leq M$, we thus have in particular decoupled $[1,2]\times [0,(10M)^{-1}]$ into $(\tau,\Gs^{l})$-flat parallelograms. By the geometry prescribed in the lemma, we may slightly extend each such parallelogram in the $y$-direction to make it become a rectangle. It is then easy to see that these rectangles satisfy the prescribed geometric properties at step $l$, and hence we have successfully established the induction step.
\end{proof}

\section{Decoupling curved neighbourhoods: Part I} \label{sec_dec_cur_p1}

In this section we deal with Case \descref{(B1)} introduced in Section \ref{sec_classification}. We again only consider the part in the first quadrant. 

If $s=r$ then the curve $x^s-\Gl_j y^r=0$ is a straight line, which by rotation reduces to Task \descref{(A)}. Thus, without loss of generality, we may assume that $s>r$. 

In this section, we fix $j$ and write $\lambda=\lambda_j>0$. We define $\gamma:=\{(x,y):y=\lambda^{-1/r}x^{s/r},x>0\}$. By abuse of notation, we also write $\gamma(x) =\lambda^{-1/r} x^{s/r}$.  Given intervals $I$ and $J$, denote 
$$R(I,J):=\{(x,y)\in \R^2 : x \in I, y-\gamma(x) \in J\}.$$ Our task is to prove the following:

\begin{prop}\label{prop_sc4}
Let $\phi(x,y)$ be a mixed-homogeneous polynomial with a factor $(x^s-\Gl y^r)^2$.  Then $R([1,2],[0,c_{\phi}])$ can be decoupled into $(\phi,\delta)$-flat rectangles. 
\end{prop}

The decoupling of $R([1,2],[-c_{\phi},0])$ is similar.

\subsection{Dyadic decomposition} Let $k \geq 2$ be the largest integer such that $(x^s-\lambda y^r)^k$ divides $\phi$. By Proposition \ref{prop_curve_2k-3}, we have $|\det D^2 \phi| \sim (x^s-\lambda y^r)^{2k-3}$ on $R([1,2],[0,c_{\phi}])$ for small enough $c_\phi$. We ``dyadically'' decompose $R([1,2],[0,c_{\phi}])$ as follows. Let $C_{rs}$ be a constant depending only on $r,s,\lambda$ to be determined and 
\begin{equation}\label{eqn_c_dya}
    c_{dya} = 1+\frac{1}{C_{rs}}.
\end{equation}
Let $R_0 = R([1,2],[0,\delta^{\frac{1}{k}}])$. Define
$$
R_j = R([1,2], [c_{dya}^{j-1}\delta^{\frac{1}{k}},c_{dya}^{j}\delta^{\frac{1}{k}}]).
$$
It suffices to decouple each $R_j$ since we can can afford logarithmic losses. The choice of the power of $\delta$ will be clear later.

\subsection{Preliminary decoupling} We will first decouple $R_j, j\geq1$. Denote $\sigma = c_{dya}^{j-1}\delta^{\frac{1}{k}}/{C_{rs}}$ We start by decoupling the curved neighbourhood into rectangles by the following proposition.

\begin{prop}\label{prop_curve_neig_dec}
Let $\phi(x,y)$ be a smooth function. Then $R([1,2],[0,\sigma])$ can be decoupled into curved regions $\{R(I,[0,\sigma]):I\in \mathcal I\}$, where $\mathcal I$ is defined by
$$
\mathcal I=\{[1+j'\sigma^{1/2},1+(j'+1)\sigma^{1/2}],j'=0,...,\sigma^{-1/2}\}.
$$
\end{prop}
The following simple proof works regardless of the function $\phi$ and the scale $\Gd$. It manifests a general ``projection property" of decoupling.
\begin{proof}
We will prove the $l^2$-inequality first. Let $f$ be Fourier supported on $\mathcal N^\phi_{\Gd}(R([1,2]\times [0,\Gs]))$. By definition
$$
\norm {f}_{L^4(\R^3)}=\left(\iiint |f(u,v,w)|^4 dudvdw\right)^{1/4}. 
$$
For each $w$, write $g^{(w)}(u,v):=f(u,v,w)$. Then $\widehat
{g^{(w)}}$ is supported on $R([1,2]\times [0,\Gs])\sub \R^2$. Since the curve $\gamma$ has nonvanishing curvature, we may apply $l^2$-lower dimensional decoupling to get
\begin{equation}\label{eqn_projection}
    \left(\iint |g^{(w)}(u,v)|^4 dudv\right)^{\frac 1 4}\lesssim_\Ge \Gs^{-\Ge} \left(\sum_{I\in \mathcal I}\norm{g_I^{(w)}}^2_{L^4(\R^2)}\right)^{\frac 1 2},
\end{equation}
where the implicit constant is independent of $w$, and $g_I^{(w)}$ denotes the Fourier restriction of $g^{(w)}$ to the strip $I\times \R$. But by direct computation,
\begin{align*}
    \widehat {g_I^{(w)}}(x,y)
    &:=\widehat g^{(w)}(x,y)1_I(x)\\
    &=1_I(x)\iint  f(u,v,w)e^{-2\pi i(ux+vy)} dudv \\
    &=1_I(x) \int \hat f(x,y,z)e^{2\pi i w z}dz\\
    &= 1_{R(I,[0,\Gs])}(x,y)\int \hat f(x,y,z)e^{2\pi i w z}dz,
\end{align*}
where in the last line we have used the Fourier support of $f$. This proves that 
$$
g_I^{(w)}=f_{R(I,[0,\Gs])}(u,v,w),
$$
where $f_{R(I,[0,\Gs])}$ denotes the Fourier restriction of $f$ to the strip $R(I,[0,\Gs])\times \R$.

We then compute
\begin{align*}
    \norm{f}_{L^4(\R^3)}
    &\lesssim_\Ge \Gs^{-\Ge}\left(\int\left(\sum_{I\in \mathcal I}\norm{g_I^{(w)}}^2_{L^4(\R^2)}\right)^{\frac 1 2 \cdot 4} dw\right)^{\frac 1 4}\\
    (\text{by Minkowski's inequality})&\leq \Gs^{-\Ge}\left(\sum_{I\in \mathcal I}\left(\int\norm{g_I^{(w)}}^4_{L^4(\R^2)} dw\right)^{\frac 1 4 \cdot 2}\right)^{\frac 1 2}\\
    &=\Gs^{-\Ge}\left(\sum_{I\in \mathcal I}\norm{f_{R(I,[0,\Gs])}}_{L^4(\R^3)}^2\right)^{\frac 1 2}.
\end{align*}
The $l^4$-inequality is even easier, as can be seen by integrating the 4th power of \eqref{eqn_projection} in the $w$ variable.

\end{proof}

By Proposition \ref{prop_curve_neig_dec}, $R_j$ can be decoupled into $R([1+j'\sigma^{1/2},1+(j'+1)\sigma^{1/2}],[C_{rs}\sigma,(C_{rs}+1)\sigma])$, $1\leq j'\leq \Gs^{-1/2}$. Write $x_0 =1+j'\sigma^{1/2} \in [1,2]$. It suffices to decouple $R([x_0,x_0+\sigma^{1/2}],[C_{rs}\sigma,(C_{rs}+1)\sigma])$. As in Figure \ref{fig:approx_parallelogram} below, we approximate the region by a parallelogram $\Tilde{R}$ given by
$$
\tilde R:=\{ (x_0, \gamma(x_0)+C_{rs}\sigma) + u(0,1) + v(1,\gamma'(x_0)) : 0\leq u \leq C_{rs}\sigma, 0 \leq v \leq \sigma^{1/2}\}.
$$

\begin{figure}[h]
    \centering
        \begin{tikzpicture}
        \draw (-0.2,0) -- (0.2,0);
        \draw[densely dotted] (0.2,0) -- (0.7,0);
        \draw[->] (0.7,0) -- (6,0);
        \draw (0,-0.2) -- (0,0.2);
        \draw[densely dotted] (0,0.2) -- (0,0.7);
        \draw[->] (0,0.7) -- (0,6);
        \node at (-0.25,-0.25) {\small $O$};
        \node at (6,-0.2) {\small $x$};
        \node at (-0.2,6) {\small $y$};
        
        \draw (1,-0.1) -- (1,0.1);
        \node at (1,-0.3) {\footnotesize $x_0$};
        \draw[dashed] (1,0) --(1,1);
        
        \draw (4,-0.1) -- (4,0.1);
        \node at (4.1,-0.25) {\footnotesize $x_0+\sigma^{1/2}$};
        \draw[dashed] (4,0) --(4,2+7/8);
        
        \draw (-0.1,1) -- (0.1,1);
        \node at (-0.55,0.98) {\scriptsize $\gamma(x_0)$};
        \draw[dashed] (0,1) --(1,1);
        
        \draw (-0.1,1.25) -- (0.1,1.25);
        \node at (-0.85,1.27) {\scriptsize $\gamma(x_0)+\sigma$};
        \draw[dashed] (0,1.25) --(1,1.25);
        
        \draw (-0.1,2.5) -- (0.1,2.5);
        \node at (-1.05,2.48) {\scriptsize $\gamma(x_0)+C_{rs}\sigma$};
        \draw[dashed] (0,2.5) --(1,2.5);
        
        \draw (-0.1,2.75) -- (0.1,2.75);
        \node at (-1.45,2.77) {\scriptsize $\gamma(x_0)+(C_{rs}+1)\sigma$};
        \draw[dashed] (0,2.75) --(1,2.75);
        
        \draw (-0.1,4) -- (0.1,4);
        \node at (-1.15,4) {\scriptsize $\gamma(x_0)+2 C_{rs}\sigma$};
        \draw[dashed] (0,4) --(1,4);
        
        \draw plot[smooth, domain=0.7:4] (\x, \x * \x /8 + 7/8);
        \draw plot[smooth, domain=0.7:4] (\x, \x /4 + 3/4);
        \draw[->] (4.2,2.5) -- (4.2,2+7/8);
        \draw[->] (4.2,2.5) -- (4.2,1.75);
        \node at (5,2.32) {\footnotesize $\sim \gamma''(x_0) \sigma$};
        \draw[->,>=stealth] (4.2,1) to [out=180,in=270] (3, 9/8+7/8-0.05);
        \node at (5,1) {\footnotesize $y=\gamma(x)$};
        \draw[->,>=stealth] (4.2,0.5) to [out=180,in=270] (2.5,5/8+3/4-0.05);
        \node at (6.37,0.5) {\footnotesize $y=\gamma'(x_0)(x-x_0)+\gamma(x_0)$};
        
        \fill[gray] plot[smooth, domain=1:4] (\x, \x * \x /8 + 7/8+1.5) -- (4,2 + 7/8+1.75) -- plot[smooth, domain=4:1] (\x, \x * \x /8 + 7/8+1.75) -- (1,2.5);
        
        \draw[thick] plot[smooth, domain=1:4] (\x, \x * \x /8 + 7/8+0.25);
        \draw[thick] plot[smooth, domain=1:4] (\x, \x * \x /8 + 7/8+3);
        \draw[thick] (1,1.25) -- (1,4);
        \draw[thick] (4,2 + 7/8+0.25) -- (4,2 + 7/8+3);
        
        \draw[] plot[smooth, domain=1:4] (\x, \x /4 + 3/4+1.5);
        \draw[] plot[smooth, domain=1:4] (\x, \x /4 + 3/4+3);
        \draw[] (1,2.5) -- (1,4);
        \draw[] (4,1+3/4+1.5) -- (4,1+3/4+3);
    \end{tikzpicture}
    \caption{Approximation by parallelograms}
    \label{fig:approx_parallelogram}
\end{figure}
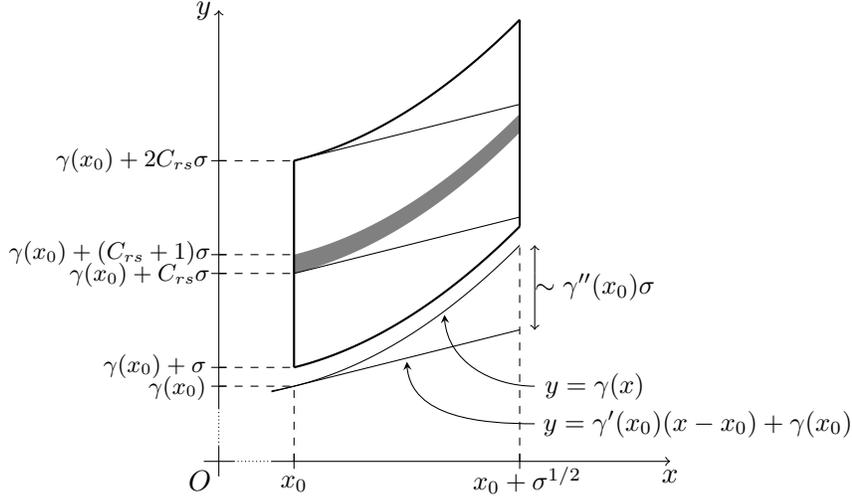
By choosing 
\begin{equation}\label{eqn_C_rs}
    C_{rs}=\sup_{1\leq x_0 \leq 2} \gamma''(x_0) +1,
\end{equation}
we see that $\Tilde{R}$ contains $R([x_0,x_0+\sigma^{1/2}],[C_{rs}\sigma,(C_{rs}+1)\sigma])$ and is contained in $R([x_0,x_0+\sigma^{1/2}],[\sigma,2C_{rs}\sigma])$. Thus, it suffices to decouple $\Tilde{R}$, on which we have $|\det D^2 \phi| \sim \sigma^{2k-3}$.

\subsection{Shear transform, translation and rescaling} To decouple $\Tilde{R}$, we first do a shear transformation that sends $\Tilde{R}$ to an axis parallel rectangle. Define 
$$
T_{x_0}(x,y) := (x, y+ \gamma(x_0) + \gamma' (x_0)(x-x_0))
$$
and
\begin{equation}\label{eqn_psi_x0}
    \psi^{(x_0)}(x,y) = \phi \circ T_{x_0}(x,y) = \phi(x, y+ \gamma(x_0) + \gamma'(x_0)(x-x_0)).
\end{equation}
The pre-image of $\Tilde{R}$ under $T$ is exactly $[x_0,x_0+\sigma^{1/2}] \times [C_{rs} \sigma, 2 C_{rs} \sigma]$, which we denote by $R_{axis}^{(x_0)}(\sigma)$. Our task is reduced to decoupling $R_{axis}^{(x_0)}(\sigma)$ into $(\psi^{(x_0)},\delta)$-flat rectangles. To do this, we need the following estimates:

\begin{lem} \label{lem_sc4_estimates}
Let $\psi^{(x_0)}$ be as in \eqref{eqn_psi_x0} above. Then for any $(x,y) \in R_{axis}^{(x_0)}(\sigma)$, we have
$$
|\psi^{(x_0)}_{xx}(x,y)| \lesssim \sigma^{k-1}, \quad |\psi^{(x_0)}_{yy}(x,y)| \sim \sigma^{k-2}, \quad |\det D^2 \psi^{(x_0)}(x,y)| \sim \sigma^{2k-3}.
$$
\end{lem}

Suppose that we have shown the above estimates. We define
$$
\psi(x,y) = \sigma^{-k}\psi^{(x_0)}(\sigma^{1/2} x+x_0, C_{rs}\sigma y + C_{rs}\sigma) - \text{linear terms}.
$$

Then, $|\det D^2 \psi| \sim 1$, $|\psi_{xx}(x,y)| \lesssim 1$ and $|\psi_{yy}(x,y)| \sim 1$ over $[0,1]^2$. Since $\psi$ is a polynomial, this implies that $\psi$ has bounded $C^3$ norm. Therefore, Theorem \ref{thm_Bourgain_Demeter} implies that $[0,1]^2$ can be decoupled into $(\psi,\delta \sigma^{-k})$-flat rectangles. Transforming back, we see that $R_{axis}^{(x_0)}(\sigma)$ can be decoupled into $(\psi^{(x_0)},\delta)$-flat rectangles and hence $R_j$ into $(\phi,\delta)$-flat rectangles.

We are left with decoupling $R_0$ into $(\phi,\delta)$-flat rectangles. By Proposition \ref{prop_curve_neig_dec}, $R_0$ can be decoupled into the pieces given by $R([1+j'\delta^{\frac{1}{2k}},1+(j'+1)\delta^{\frac{1}{2k}}],[0,\delta^{\frac{1}{k}}])$ for a various $j'$. We claim that these pieces are already $(\phi,\delta)$-flat. To see this, write $x_0 =1+ j'\delta^{\frac{1}{2k}}$. By the same shear transformation $T_{x_0}$, it suffices to show that $R^{(x_0)}_{0}(\delta^{\frac{1}{k}}):=[x_0,x_0+\delta^{\frac{1}{2k}}] \times [0,\delta^{\frac{1}{k}}]$ is  $(\psi^{(x_0)},\delta)$-flat. By Lemma \ref{lem_sc4_estimates}, we have for any $(x,y) \in R^{(x_0)}_{0}$,
$$
|\psi^{(x_0)}_{xx}(x,y)| \lesssim \delta^{\frac{k-1}{k}}, \quad |\psi^{(x_0)}_{yy}(x,y)| \lesssim \delta^{\frac{k-2}{k}}, \quad |\psi^{(x_0)}_{xy}(x,y)| \lesssim \delta^{\frac{k-3/2}{k}}.
$$
Define
$$
\Tilde{\psi}(x,y) = \psi^{(x_0)}(x+x_0,y) -\text{linear terms}.
$$
It is clear that for $0\leq x \leq \delta^{\frac{1}{2k}}$, $0 \leq y \leq \delta^{\frac{1}{k}}$, we have $|\Tilde{\psi}|\lesssim \delta$. This implies that $R^{(x_0)}_{0}$ is $(\psi^{(x_0)},\delta)$-flat. Transforming back, we see that $R([x_0,x_0+\delta^{\frac{1}{2k}}],[0,\delta^{\frac{1}{k}}])$ is $(\phi,\delta)$-flat. 

Now, assuming Lemma \ref{lem_sc4_estimates}, we have shown Proposition \ref{prop_sc4}.

\subsection{The proof of Lemma \ref{lem_sc4_estimates}}

\begin{proof}
The proof is divided into two steps. The following lemma is helpful in both steps.
\begin{lem}\label{lem:non-linear shear}
Let $f(x,y) = g(x,y+h(x))$. Then
$$
f_y =g_y, \quad f_x = g_x + h'(x) g_y, \quad f_{yy} = g_{yy}
$$
$$
f_{xy} = g_{xy} + h'(x) g_{yy}, \quad f_{xx} = g_{xx} + 2h'(x) g_{xy} +(h'(x))^2g_{yy}+ h''(x)g_{y}.
$$
In particular, if $h$ is linear, then $\det D^2f = \det D^2 g$.
\end{lem}
\begin{proof}
By direct computation.
\end{proof}

As a direct corollary of Lemma \ref{lem:non-linear shear}, we have $|\det D^2 \psi^{(x_0)}| \sim \sigma^{2k-3}$ on $R^{(x_0)}_{axis}(\sigma)$.

In the first step, let us consider the line segment $x=x_0$, $C_{rs}\Gs\leq y\leq 2C_{rs}\Gs$.

\subsubsection{On the line segment $(x_0,y)$} We define an auxiliary function as follows:
\begin{equation}\label{eqn:aux_fcn}
    \psi^{aux}(x,y) = \phi(x,y+\gamma(x)).
\end{equation}
Then we have
$$
\psi^{(x_0)}(x,y) = \psi^{aux}(x,y-\left(\gamma(x) -\gamma(x_0)\right) + \gamma'(x_0)(x-x_0)).
$$
Applying Lemma \ref{lem:non-linear shear},
$$
\psi^{(x_0)}_{xy}(x_0,y) = \psi^{aux}_{xy}(x_0,y), \quad \psi^{(x_0)}_{yy}(x_0,y) = \psi^{aux}_{yy}(x_0,y).
$$
On the other hand, $(x^s -\lambda y^r)^k$ divides $\phi$. This implies that $\phi$ can be written as $(y-\gamma(x))^kP(x^{1/r},y)$ for some mixed homogeneous polynomial $P$ that doesn't vanish on $\gamma$. Then $\psi^{aux} = y^k Q(x^{1/r},y)$ where $Q$ is a mixed homogeneous polynomial with coefficients independent of $\sigma$, and $Q$ doesn't vanish near the $x$-axis. Therefore, $|\psi^{aux}_{yy}(x,y)| \sim |y^{k-2}Q(x,y)| \sim \sigma^{k-2}$ and $|\psi^{aux}_{xy}(x,y)|\lesssim |y^{k-1}\partial_{x}(Q(x^{1/r},y))| \lesssim \sigma^{k-1} $ when $x\in [1,2]$ and $y\sim \sigma$. In summary, we have
$$
|\psi^{(x_0)}_{xy}(x_0,y)| \lesssim \sigma^{k-1}, \quad |\psi^{(x_0)}_{yy}(x_0,y)| \sim \sigma^{k-2}.
$$
Finally, using $|\det D^2 \psi^{(x_0)}| \sim \sigma^{2k-3}$, we have $|\psi^{(x_0)}_{xx}(x_0,y)| \sim \sigma^{k-1}$. Therefore, we have shown the desired estimates on the line segment $(x_0,y)$, $y\sim \sigma$.

\subsubsection{The remaining region} Let $(x_1,y_1) \in R^{(x_0)}_{axis}(\sigma)$. By applying the result above to 
$$
\psi^{(x_1)}(x,y) = \phi(x,y+\lambda^{-1/r}x_1^{s/r} + \gamma'(x_1)(x-x_1)),
$$
we have for $y \sim \sigma$,
$$
|\psi^{(x_1)}_{xx}(x_1,y)| \sim \sigma^{k-1},\quad |\psi^{(x_1)}_{xy}(x_1,y)| \lesssim \sigma^{k-1}, \quad |\psi^{(x_1)}_{yy}(x_1,y)| \sim \sigma^{k-2}.
$$
Now,
$$
\psi^{(x_0)}(x,y) = \psi^{(x_1)}\left(x,y - (\gamma'(x_1)-\gamma'(x_0))(x-x_1)+\gamma'(x_0)(x_1-x_0)- (\gamma(x_1)-\gamma(x_0))\right).
$$
From Figure \ref{fig:approx_parallelogram} and the way we have chosen $C_{sr}$, we have
$$
y':= y_1 +\gamma'(x_0)(x_1-x_0)-(\gamma(x_1)-\gamma(x_0)) \sim \sigma.
$$
On the other hand, $|\gamma'(x_1)-\gamma'(x_0)| \lesssim |x_1-x_0| \lesssim \sigma^{1/2}$ since $\gamma'$ is Lipschitz on $[1,2]$.
By Lemma \ref{lem:non-linear shear}, we have
$$
|\psi_{yy}^{(x_0)}(x_1,y_1)| = |\psi_{yy}^{(x_1)}(x_1,y')| \sim \sigma^{k-2};
$$
$$
|\psi_{xy}^{(x_0)}(x_1,y_1)| \leq |\psi_{xy}^{(x_1)}(x_1,y')|+ O(\sigma^{1/2})|\psi_{yy}^{(x_1)}(x_1,y')| \lesssim \sigma^{k-3/2};
$$
and
$$
    |\psi_{xx}^{(x_0)}(x_1,y_1)| \leq |\psi_{xx}^{(x_1)}(x_1,y')|+ O(\sigma^{1/2})|\psi_{xy}^{(x_1)}(x_1,y')| + O(\sigma)|\psi_{yy}^{(x_1)}(x_1,y')|\lesssim \sigma^{k-1}.
$$
This finishes the proof of Lemma \ref{lem_sc4_estimates}.
\end{proof}

\section{Decoupling curved neighbourhoods: Part II}\label{sec_dec_curve_p2}

In this section we deal with Case \descref{(B2)} introduced in Section \ref{sec_classification}. We again only consider the part in the first quadrant. Without loss of generality, we may assume that $s>r$. We also fix $j$ and write $\lambda=\lambda_j>0$.

As in the last section, define $\gamma= \{(x,y): y=\lambda^{-1/r} x^{s/r}\}$. By abuse of notation, we write $\gamma(x) =\lambda^{-1/r} x^{s/r}$. Given intervals $I$ and $J$, denote $R(I,J)$ to be the set $\{(x,y)\in \R^2 : x \in I,y-\gamma(x) \in J\}.$

Our task is to prove the following:

\begin{prop}\label{prop_sc5}
Let $\phi(x,y)$ be a mixed-homogeneous polynomial that cannot be divided by $(x^s-\lambda y^r)^2$. Then $R([1,2],[0,c_{\phi}])$ can be decoupled into $(\phi,\delta)$-flat rectangles. 
\end{prop}

The decoupling of $R([1,2],[-c_{\phi},0])$ is similar. The proof of Proposition \ref{prop_sc5} relies heavily on the techniques developed in Sections \ref{sec_dec_axi_p2} and \ref{sec_dec_cur_p1}. 

\subsection{Reduction to $R_{axis}$} Let $k \geq 1$ be the largest integer such that $(x^s - \lambda y^r)^k$ divides $\det D^2 \phi$. Let $C_{rs}$ and $c_{dya}$ be as in \eqref{eqn_C_rs} and \eqref{eqn_c_dya}, respectively. Let $R_0 = R([1,2],[0,\delta^{\frac{1}{k+2}}])$. Define
$$
R_j = R([1,2], [c_{dya}^{j-1}\delta^{\frac{1}{k+2}},c_{dya}^{j}\delta^{\frac{1}{k+2}}]).
$$
It suffices to decouple each $R_j$ since we can can afford logarithmic losses. The choice of the power of $\delta$ will be clear later.

We will first decouple $R_j, j\geq 1$. Denote $\sigma = c^{j-1}_{dya}\delta^{\frac{1}{k+2}}/C_{rs}$. By Proposition \ref{prop_curve_neig_dec}, $R_j$ can be decoupled into curved regions $\{R([1+j'\sigma^{1/2},1+(j'+1)\sigma^{1/2},[C_{rs}\sigma,(C_{rs}+1)\sigma])$: $1\leq j'\leq \Gs^{-1/2}$\}. Referring to Figure \ref{fig:approx_parallelogram}, this can be approximated by a parallelogram $\Tilde{R}$, on which we can do a shear transformation $T_{x_0}^{-1}$ to $R_{axis}^{(x_0)}(\sigma) := [x_0,x_0+\sigma^{1/2}] \times [C_{rs}\sigma,2C_{rs}\sigma]$. It suffices to decouple $R_{axis}^{(x_0)}(\sigma)$ into $(\psi^{(x_0)},\delta)$-flat rectangles. 

By the similar argument, to decouple $R_0$, it suffices to decouple $R^{(x_0)}_{0}(\delta^{\frac{1}{k+2}}):= [x_0,x_0+\delta^{\frac{1}{2(k+2)}}] \times [0,\delta^{\frac{1}{k+2}}]$ into $(\psi^{(x_0)},\delta)$-flat rectangles.

\subsection{Transform to a function in the family $\mathcal{A}_l$} We postpone the proof of the following estimate to the next subsection:

\begin{lem}\label{lem_sc5_estimates} Let $\psi^{(x_0)}$ be as in \eqref{eqn_psi_expand} above. Then for any $(x,y) \in R_{axis}^{(x_0)}(\sigma)$, we have $|\psi_{xx}^{(x_0)}| \sim 1$.
\end{lem}

Assuming Lemma \ref{lem_sc5_estimates} for now, we define
$$
T_\sigma(x,y)  = (\sigma^{1/2}(x-1)+x_0, C_{rs}\sigma y+ C_{rs} \sigma)
$$
so that $T_\sigma([1,2]\times [0,1]) = R_{axis}^{(x_0)}(\sigma)$. Define
$$
\psi(x,y) = \sigma^{-1} (\psi^{(x_0)}\circ T_\sigma)(x,y) - \text{linear terms}.
$$

Now, on the set $[1,2]\times [0,1]$, $|\psi_{xx}| \sim 1$, $|\psi_{xy}| \lesssim \sigma^{1/2}$ and $|\det D^2 \psi| \sim \sigma^{k+1}$. We see that $\psi \in \mathcal{A}_{2k+2}(\sigma^{1/2})$\footnote{Technically, the properties over a slightly enlarged parallelogram are required. We omit the details here. Readers may refer to Section \ref{sec_dec_axi_p2}.}. Note that the choice of exponent in the dyadic decomposition implies that $\delta \sigma^{-1}\leq (\sigma^{1/2})^{2k+2}$. By Proposition \ref{prop_R_j}, we obtain a decoupling of $[1,2]\times [0,1]$ into $(\psi,\delta \sigma^{-1})$-flat rectangles. Under the scaling, this gives a decoupling of $R_{axis}^{(x_0)}(\sigma)$ into $(\psi^{(x_0)},\delta)$-flat rectangles and a decoupling of $R_j$ into $(\phi,\delta)$-flat rectangles. 

We now decouple $R_0$. By Proposition \ref{prop_curve_neig_dec}, we can decouple $R_0$ into curved regions $R([x_0,x_0+\delta^{\frac{1}{2(k+2)}}],[0,\delta^{\frac{1}{k+2}}])$ for a family of $x_0$'s. Applying the same shear transform $T_{x_0}$, we obtain $R_0^{(x_0)}(\delta^{\frac{1}{k+2}}):=R([x_0,x_0+\delta^{\frac{1}{2(k+2)}}],[0,\delta^{\frac{1}{k+2}}])$. We define
$$
T_0(x,y) = (\delta^{\frac{1}{2(k+2)}}x + x_0 , \delta^{\frac{1}{k+2}} y)
$$
so that $T_0([1,2] \times [0,1]) = R_{0}^{(x_0)}(\delta^{\frac{1}{k+2}})$. Define
$$
\Tilde{\psi}(x,y) = \delta^{-\frac{1}{k+2}}(\psi^{(x_0)} \circ T_\sigma)(x,y)  -\text{linear terms}.
$$
Similarly, we see that $\Tilde{\psi} \in \mathcal{A}_{2k+2}(\delta^{\frac{1}{2(k+2)}})$. Lemma \ref{lem_cylinderial_dec} gives a decoupling of $[1,2] \times [0,1]$ into $(\Tilde{\psi},\delta^{\frac{k+1}{k+2}})$-flat rectangles. Reversing all the transformation, we obtain a decoupling of $R_0$ into $(\phi,\delta)$-flat rectangles. This completes the proof of Proposition \ref{prop_sc5} assuming Lemma \ref{lem_sc5_estimates}.

\subsection{Proof of Lemma \ref{lem_sc5_estimates}}\label{sec_sc5_estimates}

\begin{proof}
Since $\psi_{xx}^{(x_0)}$ is Lipschitz and $R_{axis}^{(x_0)}(\sigma)$ has diameter $O(\sigma^{1/2})$, it suffices to show that $\psi_{xx}^{(x_0)}(x_0,0)\neq 0$. 
Suppose on the contrary that $\psi_{xx}^{(x_0)}(x_0,0)=0$. Since $\det D^2 \psi^{(x_0)}(x_0,0) = \det D^2 \phi(x_0,\gamma(x_0))=0$, we see that $\psi_{xy}^{(x_0)}(x_0,0)=0$. Consider the auxiliary function in \eqref{eqn:aux_fcn},
$$
    \psi^{aux}(x,y) = \phi(x,y+\gamma(x)).
$$
Lemma \ref{lem:non-linear shear} gives $\psi^{aux}_{xy}(x_0,0) = \psi^{(x_0)}_{xy}(x_0,0) =0$.

Since $\phi$ cannot be divided by $(x^s-\lambda y^r)^2$, one of the following cases happens:
\begin{enumerate}
    \item $(y^r - \lambda^{-1}x^s)$ divides $\phi$ but $(y^r - \lambda^{-1}x^s)^2$ doesn't divide $\phi$. By Proposition \ref{prop_sc5_estimates}, we have
    $$
        \psi^{aux}(x,y) = c_1 x^{s(r-1)/r + m}y  + c_2 y^2 P(x^{1/r},y).
    $$
    for some $c_1 \neq 0$, $m\geq 0$, and some polynomial $P$.
    
     \item $(y^r - \lambda^{-1}x^s)$ doesn't divide $\phi$.  By Proposition \ref{prop_sc5_estimates}, we have
    $$
    \psi^{aux}(x,y) = c_0 x^{s+m} + c_1  x^{s(r-1)/r+m}y + c_2 y^2 Q(x^{1/r},y).
    $$
    for some $c_0 \neq 0$, $m\geq 0$, and some polynomial $Q$.
    
\end{enumerate}

For the first case, $\psi^{aux}_{xy}(x_0,0)=0$ implies $s(r-1)/r + m=0$, which implies to $m=0$ and $r=1$. Hence, $\phi$ is for the form $c_1 (y-\lambda^{-1}x^s)$ and 
$$\psi^{(x_0)}(x,y)=c_1 (y-\gamma(x)+ \gamma'(x_0)(x-x_0)).$$

Recall that $\gamma(x) = \lambda^{-1}x^{s/r}$. $\psi^{(x_0)}_{xx} = 0$ implies that $\gamma''(x_0) = 0$. Therefore, $s=1$, which contradicts the assumption that $r\neq s$.

For the second case, $\psi^{aux}_{xy}(x_0,0)=0$ implies either $s(r-1)/r + m=0$ or $c_1=0$. Similarly, $s(r-1)/r + m=0$ implies that $\phi$ is of the form $c_0 x^s+ c_1 (y-\lambda^{-1}x^s)$, which leads to a contradiction similar to the above. For $c_1=0$, $\psi^{aux}_{y}(x_0,0)=0$ and hence by Lemma \ref{lem:non-linear shear},
$$
\psi^{(x_0)}_{xx}(x_0,0) = \psi^{aux}_{xx}(x_0,0) \neq 0,
$$
from which a contradiction arises.

In summary, both cases lead to contradictions. Therefore, we have $\psi^{(x_0)}_{xx}(x_0,0) \neq 0$, and thus we have completed the proof of Lemma \ref{lem_sc5_estimates}.
\end{proof}

\section{The convex case}\label{sec_convex}
In this section we briefly mention why we can have $\ell^2(L^4)$ decoupling estimates as in \eqref{eqn_decoupling_l2} when $\phi$ is convex. First, when $\phi$ has positive Gaussian curvature we can invoke \eqref{eqn_BD_decoupling_l2} to strengthen the $\ell^4(L^4)$ estimates to $\ell^2(L^4)$ estimates. Second, notice that throughout this paper, if we start with a convex polynomial $\phi$, then all its ``descendants" for which decouplings are studied are composition of $\phi$ with an affine map. By ``descendants" here we mean, for instance, $\tilde \phi$ in \eqref{eqn_tilde_phi}, $\eta$ in \eqref{eqn_defn_eta} and $\psi^{(x_0)}$ in \eqref{eqn_psi_x0}. Since the composition of a convex function with an affine map is still convex, \eqref{eqn_BD_decoupling_l2} is still applicable. Thus we have the $\ell^2(L^4)$ inequalities in the convex case.

\section{Appendix}\label{sec_appendix}

\subsection{Facts about mixed-homogeneous polynomials}

First of all, mixed-homogeneous polynomials and their Hessian determinants share a similar homogeneity:

\begin{prop}
Suppose that $\phi:\R^2 \to \R$ is a mixed homogeneous polynomial, that is, for some positive integers $q,r,s$ with $\mathrm{gcd}(r,s)=1$ we have
\begin{equation}\label{eqn_defn_mixed_homo}
  \phi(x,y) = \rho^{-q} \phi(\rho^r x, \rho^s y),\quad \text{for all $(x,y)\in \R^2$ and all $\rho>0$}.   
\end{equation}
Then the Hessian determinant of $\phi$ satisfies
$$
(\det D^2 \phi)(x,y) = \rho^{-2(q-(r+s))} (\det D^2 \phi)(\rho^r x ,\rho^s y), \quad \text{for all $(x,y)\in \R^2$ and all $\rho>0$}. 
$$
\end{prop}

\begin{proof}
Direct computation.
\end{proof}

Next, we need the following lemma.
\begin{lem}[\cites{DZ2019,IM2011}]
Every non-zero real mixed-homogeneous polynomial $\phi$ satisfying \eqref{eqn_defn_mixed_homo} has the following factorization property:
\begin{equation}\label{eqn_factorisation_intro}
   \phi(x,y)=x^{\nu_1}y^{\nu_2}\prod_{j=1}^{M} (x^s - \lambda_j y^r)^{n_j}P(x^s,y^r),
\end{equation}
for some nonnegative integers $\nu_1,\nu_2,M,n_j$, nonzero real numbers $\Gl_j$, and a real homogeneous polynomial $P$ that never vanishes except at the origin.

As a result, $\phi$ admits the following expansion:
\begin{equation}\label{eqn_expansion}
    \phi(x,y)=x^{\nu_1}y^{\nu_2}(c_0 x^{ns} +c_1 x^{(n-1)s}y^r+c_2 x^{(n-2)s}y^{2r}+\cdots+c_n y^{nr}).
\end{equation}
\end{lem}

\begin{proof}
By Proposition 2.2 in \cite{IM2011}, there exists nonnegative integers $\nu_1$, $\nu_2$, $n_j$, $n_{j'}$, $M$, $M'$, real numbers $C$, $\lambda_j$ and complex numbers $\lambda_{j'}'\not \in \mathbb{R}$ such that
$$
\phi(x,y) = C x^{\nu_1}y^{\nu_2}\prod_{j=1}^M(x^s-\lambda_j y^r)^{n_j} \prod_{j'=1}^{M'}(x^s-\lambda_j' y^r)^{n_j'}.
$$
Let $P(u,v) = C\prod_{j'=1}^{M'}(u-\lambda_j' v)^{n_j'}$. It is clear that $P$ is a homogeneous polynomial. Since $\phi$ is a real polynomial and is the product of $P(x^s,y^r)$ and some real polynomials, $P$ must also be a real polynomial. To show that $P$ never vanishes except at the origin, we first note that if $P(1,0) = 0$, then $C=0$ and $\phi$ is the zero polynomial, a contradiction. So $P(1,0) \neq 0$ and thus $C\neq 0$. Also, $P(c,1) \neq 0$ for all $c\in \R$. But recall that $P$ is a homogeneous polynomial, and thus its zero set is a collection of straight lines passing through the origin. This shows that $P$ can only vanish at the origin.
\end{proof}

The following proposition helps us in the classification of scenarios in Section \ref{sec_classification}.

\begin{prop}\label{prop_y2_not_divide}
Let $\phi$ be a mixed-homogeneous polynomial without linear terms such that $y^2$ does not divide $\phi$. If $\det(D^2\phi)$ vanishes on the $x$-axis, then $\phi$ is of the form
$$
\phi(x,y)=Cx^m+yP(x,y)
$$
for some $C\neq 0$, $m\geq 2$ and some polynomial $P$.
\end{prop}
\begin{proof}
Using \eqref{eqn_expansion}, we expand
$$
\phi(x,y)=x^{\nu_1}y^{\nu_2}(c_0 x^{ns} +c_1 x^{(n-1)s}y^r+c_2 x^{(n-2)s}y^{2r}+\cdots+c_n y^{nr}).
$$
Thus we need to show $\nu_2=0$ and $c_0\neq 0$. Since $\phi$ does not have a linear term, we are done.

We first prove $\nu_2=0$. Suppose towards contradiction that $\nu_2\geq 1$. If $\nu_2\geq 2$ then $y^2$ divides $\phi$, which contradicts our assumption. Thus $\nu_2=1$. Then we may rewrite
$$
\phi(x,y)=c_0x^{\nu_1+ns}y+y^2Q(x,y),
$$
for some polynomial $Q$. Direct computation shows that when $y=0$ we have 
$$
\det(D^2\phi)(x,y)=-c^2_0 (\nu_1+ns)^2x^{2(\nu_1+ns-1)}.
$$
Note we must have $\nu_1+ns\geq 1$, otherwise $\phi(x,y)=c_0 y$, a linear function, contradicting the assumption. Thus, using $\det(D^2\phi)=0$ on the $x$-axis, we must have $c_0=0$. This implies that $y^{\nu_2+r}=y^{1+r}$ divides $\phi$. But since $r\geq 1$, we have $y^2$ divides $\phi$, again contradicting the assumption. Thus we have $\nu_2=0$.

Next we prove $c_0\neq 0$. Suppose towards contradiction that $c_0=0$ (then $n\geq 1$, otherwise $\phi\equiv 0$). Then 
$$
\phi(x,y)=x^{\nu_1}y^r(c_1x^{(n-1)s}+c_2 x^{(n-2)s}y^r+\cdots+c_n y^{(n-1)r}).
$$
If $r\geq 2$ then $y^2$ divides $\phi$ which is a contradiction. Since $r\geq 1$ we must have $r=1$. Using the same computation above, when $y=0$ we have
$$
\det(D^2\phi)(x,y)=-c^2_1 (\nu_1+(n-1)s)^2x^{2(\nu_1+(n-1)s-1)}.
$$
Now we must have $\nu_1+(n-1)s\geq 1$, otherwise $\nu_1=0$ and $n=1$, in which case $\phi(x,y)=c_1 y$, again a contradiction. Thus $\nu_1+(n-1)s\geq 1$ and this forces $c_1=0$. Thus $y^{2r}$ divides $\phi$, and since $r\geq 1$, we arrive again at a contradiction. Hence $c_0\neq 0$.
\end{proof}

This proposition is used in Section \ref{sec_dec_cur_p1}.
\begin{prop}\label{prop_curve_2k-3}
Suppose that $s\neq r$ and $\lambda_j>0$. Let $\phi(x,y) = (x^s-\lambda_j y^r)^k P(x,y) $ be a mixed-homogeneous polynomial where $k\geq 2$ and $P(t^r,\lambda_j^{1/r} t^s)\not\equiv 0$. Then $\det D^2 \phi = (x^s - \lambda_j y^r)^{2k-3}Q(x,y)$ where $Q(t^r,\lambda_j^{1/r} t^s)\not\equiv 0$.
\end{prop}

\begin{proof}
By writing $\Tilde{\phi}(x,y) = \phi(x,\lambda^{-1/r}y)$ and noting that $\det D^2 \Tilde{\phi} = \lambda^{-2/r}\det D^2 \phi$, we may assume without loss of generality that $\lambda_j=1$.

The case $s=1$ is proved in Lemma 10.1 of \cite{DZ2019}. The case $r=1$ is similar.  Now, assume that $r,s \geq 2$, direct computation shows
\begin{align*}
\phi_{xx}&=\left(x^s-y^r\right)^{k-2} \left[ (k-1) k s^2 x^{2 s-2} P +k (s-1) s x^{s-2}  \left(x^s-y^r\right)P \right. \\  & \left.\quad + 2 k s x^{s-1} \left(x^s-y^r\right) P_x + \left(x^s-y^r\right)^2P_{xx}  \right];\\
\phi_{xy}&=\left(x^s-y^r\right)^{k-2} \left[-(k-1) k r s x^{s-1} y^{r-1}  P+ k s x^{s-1}  \left(x^s-y^r\right)P_y \right. \\& \left.\quad -k r y^{r-1}  \left(x^s-y^r\right)P_{x}+ \left(x^s-y^r\right)^2P_{xy} \right ];  \\
\phi_{yy}&=\left(x^s-y^r\right)^{k-2} \left [(k-1) k r^2 y^{2 r-2} P -k (r-1) r y^{r-2}  \left(x^s-y^r\right)P \right.  \\&   \left.\quad  -2 k r y^{r-1}  \left(x^s-y^r\right)P_y + \left(x^s-y^r\right)^2P_{yy} \right ].
\end{align*}

We factor out $\left(x^s-y^r\right)^{k-2}$ from each of above, which contributes $\left(x^s-y^r\right)^{2k-4}$ to $\det D^2 \phi$. Now, the remaining terms without the factor $\left(x^s-y^r\right)$ in $\det D^2 \phi$ is
$$
\begin{vmatrix}
(k-1) k s^2 x^{2 s-2} P & -(k-1) k r s x^{s-1} y^{r-1}  P \\ -(k-1) k r s x^{s-1} y^{r-1}  P &(k-1) k r^2 y^{2 r-2} P
\end{vmatrix},
$$
which is zero. On the other hand, the coefficient of the factor $\left(x^s-y^r\right)$ is given by
\begin{align*}
    &\quad  \left((k-1) k s^2 x^{2 s-2} P\right) \left( -k (r-1) r y^{r-2}  P   -2 k r y^{r-1}  P_y \right)\\
    &+\left((k-1) k r^2 y^{2 r-2} P\right) \left(k (s-1) s x^{s-2} P + 2 k s x^{s-1}  P_x\right)\\
    &- 2\left(-(k-1) k r s x^{s-1} y^{r-1}  P\right)\left( k s x^{s-1}  P_y -k r y^{r-1}  P_{x}  \right) \\
    &=(k-1) k^2 r s y^{r-2} x^{s-2} P^2  \left(s x^s - r y^r - r s x^s + r s y^r\right)
\end{align*}

Applying the assumption on $P$ and $s\neq r$, the above quantity is non-zero over the curve $(t^r,t^s)$ when $t \neq 0$ as desired. 
\end{proof}

The following proposition is used in Section \ref{sec_sc5_estimates}.

\begin{prop}\label{prop_sc5_estimates}
Let $\phi$ be a nonzero mixed-homogeneous polynomial satisfying \eqref{eqn_defn_mixed_homo}. Then
\begin{enumerate}
    \item $y^r-\lambda^{-1} x^s$ divides $\phi$ if and only if 
    \begin{equation}\label{eq:prop_9.5}
        \phi(x,y) = (y-\lambda^{-1/r}x^{s/r})Q(x^{1/r},y)
    \end{equation}
    for some polynomial $Q$.
    \item If  $y^r-\lambda^{-1} x^s$ divides $\phi$ but $(y^r-\lambda^{-1} x^s)^2$ doesn't divide $\phi$, then 
    \begin{equation}\label{eq:prop_9.5_2}
        \phi(x,y) = c_1(y-\lambda^{-1/r} x^{s/r})x^{s(r-1)/r+m}+ (y-\lambda^{-1/r} x^{s/r})^2 R(x^{1/r},y)
    \end{equation}
for some $c_1 \neq 0$, $m\geq 0$ and some polynomial $R$.
\end{enumerate}

\end{prop}

\begin{proof}
By rescaling $x \to \lambda^{1/s} x$, it suffices to show the case where $\lambda=1$. 

We first prove (1). If $\phi(x,y)$ can be divided by $y^r- x^s$, then there exists a polynomial $Q_1(x,y)$ such that
\begin{equation}\label{eq:prop_9.5_proof}
    \phi(x,y) = (y^r-x^s) Q_1(x,y) = (y-x^{s/r})\left (\sum_{i=0}^{r-1} y^i x^{s(r-1-i)/r}\right )Q_1(x,y)
\end{equation}
Define 
\begin{equation}\label{eq:prop_9.5_def_Q}
    Q(u,v)= \left (\sum_{i=0}^{r-1} v^i u^{s(r-1-i)}\right )Q_1(u^r,v).
\end{equation}
Then we get \eqref{eq:prop_9.5} with $\lambda=1$.

On the other hand, if $\phi(x,y)$ can be written as \eqref{eq:prop_9.5} with $\lambda=1$ for some polynomial $Q$, then $v-u^s$ divides $\phi(u^r,v)$. Using \eqref{eqn_factorisation_intro}, we have
$$
\phi(x,y)=x^{\nu_1}y^{\nu_2}\prod_{j=1}^{M} (x^s - \lambda_j y^r)^{n_j}P(x^s,y^r),
$$
for some $P$ that vanishes only at the origin. Therefore, 
$$
\phi(u^r,v) = u^{\nu_1 r}v^{\nu_2}\prod_{j=1}^{M} (u^{sr} - \lambda_j v^r)^{n_j}P(u^{sr},v^r).
$$
Note that $v-u^s$ is irreducible and vanishes at $(1,1)$, and thus there exists some $1\leq j_0 \leq M$ such that $v-u^s$ divides $u^{sr} - \lambda_{j_0} v^r$. So $u^{sr} - \lambda_{j_0} v^r$ also vanishes at $(1,1)$. This implies that $n_j\geq 1$ and $\lambda_{j_0}=1$. Therefore, $-(x^s - \lambda_{j_0}y^r)=y^r-x^s$ divides $\phi(x,y)$ as required.

For (2), using \eqref{eq:prop_9.5_proof} and \eqref{eq:prop_9.5_def_Q}, we see that $Q$ is mixed-homogeneous and we have the form \eqref{eq:prop_9.5_2}, not knowing whether $c_1=0$. To show that $c_1 \neq 0$, it suffices to show that $v-u^s$ doesn't divide $Q(u,v)$. 

Suppose on the contrary that  $v-u^s$ divides $Q(u,v)$. Since $v-u^s$ vanishes at $(1,1)$ but $\sum_{i=0}^{r-1} v^i u^{s(r-1-i)}$ doesn't, $v-u^s$ divides $Q_1(u^r,v)$. i.e.
$$
Q_1(x,y) = (y-x^{s/r}) Q_2(x^{1/r},y)
$$
for some polynomial $Q_2$. Using Part (1), this implies that $y^r-x^s$ divides $Q_2$ and hence $(y^r-x^s)^2$ divides $\phi$. Thus a contradiction arises. 

\end{proof}

\subsection{$\Gd$-flatness} 
Decoupling inequalities for a smooth function $\phi$ are often formulated so that the $\Gd$-neighbourhood of the graph of $\phi$ is partitioned into almost rectangles. Here we make this notion precise.
\begin{defn}
Let $\phi:\R^2\to \R$ be smooth. We say $\phi$ is $\Gd$-flat over a set $Q\sub \R^2$, or alternatively $Q$ is $(\phi,\Gd)$-flat, if there is a constant $C$ depending on $\phi$ only, such that
\begin{equation*}
    \sup_{u,v\in Q} |\phi(v)-\phi(u)-\nabla \phi(u)(v-u)|\leq C \Gd.
\end{equation*}
\end{defn}

The following property shows that $\Gd$-flatness acts well with linear transformations, and is implicitly used throughout the argument in this article.
\begin{prop}
Let $\phi:\R^n\to \R$ be a smooth function and $T:\R^n\to \R^n$ be a linear transformation. Then a set $Q\sub \R^n$ is $(\phi,\Gd)$-flat if and only if $T(Q)$ is $(\phi \circ T^{-1},\Gd)$-flat.
\end{prop}

\begin{proof}
We prove a lemma first:
\begin{lem}
Let $\phi:\R^n\to \R$ be a smooth function and $T:\R^n\to \R^n$ be a linear transformation. Then for any $u\in \R^n$, we have
$$
\nabla (\phi\circ T)(u)=T^t \nabla \phi(Tu).
$$
\end{lem}
\begin{proof}[Proof of lemma]
Treat $T$ as a matrix. Write $(Tu)_i=\sum_j T_{ij}u_j$. Then 
$$
\phi\circ T(u)=\phi\left(\sum_j T_{1j}u_j,\cdots \sum_j T_{nj}u_j\right),
$$
and thus for each $k$, we have
$$
\partial_k (\phi\circ T)(u)=\sum_i T_{ik}\partial_i \phi(Tu).
$$
On the other hand, the $k$-th coordinate of $T^t \nabla \phi(Tu)$ is given by
$$
\sum_{j}(T^t)_{kj}(\nabla \phi(Tu))_j=\sum_j T_{jk}\partial_j \phi(Tu),
$$
the same as the left hand side.
\end{proof}
Now with the lemma, let $u,v\in T(Q)$ and write $u=Ta,v=Tb$ with $a,b\in Q$. Then
\begin{align*}
    &|\phi\circ T^{-1}(v)-\phi\circ T^{-1}(u)-\nabla (\phi\circ T^{-1})(u)\cdot (v-u)|\\
    &=|\phi(b)-\phi(a)-(T^{-1})^t \nabla \phi(T^{-1}u)\cdot(T(b-a))|\\
    &=|\phi(b)-\phi(a)-T^t (T^{-1})^t \nabla \phi(T^{-1}u)\cdot (b-a)|\\
    &=|\phi(b)-\phi(a)- \nabla \phi(T^{-1}u)\cdot (b-a)|.
\end{align*}
Therefore, $Q$ is $(\phi,\Gd)$-flat if and only if $T(Q)$ is $(\phi\circ T^{-1},\Gd)$-flat.

\end{proof}

{\bf Data Availability Statement:}

Data sharing not applicable to this article as no datasets were generated or analysed during the current study.

\bibliographystyle{empty}
\bibliography{sample}

\begin{bibdiv}
\begin{biblist}

\bib{BGLSX}{article}{
      author={Biswas, C.},
      author={Gilula, M.},
      author={Li, L.},
      author={Schwend, J.},
      author={Xi, Y.},
       title={{$\ell^2$} decoupling in {$\mathbb{R}^2$} for curves with
  vanishing curvature},
        date={2020},
        ISSN={0002-9939},
     journal={Proc. Amer. Math. Soc.},
      volume={148},
      number={5},
       pages={1987\ndash 1997},
         url={https://doi.org/10.1090/proc/14954},
      review={\MR{4078083}},
}

\bib{BD2015}{article}{
      author={Bourgain, J.},
      author={Demeter, C.},
       title={The proof of the {$l^2$} decoupling conjecture},
        date={2015},
        ISSN={0003-486X},
     journal={Ann. of Math. (2)},
      volume={182},
      number={1},
       pages={351\ndash 389},
         url={https://doi.org/10.4007/annals.2015.182.1.9},
      review={\MR{3374964}},
}

\bib{BD2017}{article}{
      author={Bourgain, J.},
      author={Demeter, C.},
       title={Decouplings for curves and hypersurfaces with nonzero {G}aussian
  curvature},
        date={2017},
        ISSN={0021-7670},
     journal={J. Anal. Math.},
      volume={133},
       pages={279\ndash 311},
         url={https://doi.org/10.1007/s11854-017-0034-3},
      review={\MR{3736493}},
}

\bib{BDK2019}{article}{
      author={Bourgain, J.},
      author={Demeter, C.},
      author={Kemp, D.},
       title={Decouplings for real analytic surfaces of revolution},
        date={2019},
     journal={arXiv:1908.07053},
}

\bib{CKZ2013}{article}{
      author={Carbery, A.},
      author={Kenig, C.~E.},
      author={Ziesler, S.~N.},
       title={Restriction for homogeneous polynomial surfaces in
  {$\Bbb{R}^3$}},
        date={2013},
        ISSN={0002-9947},
     journal={Trans. Amer. Math. Soc.},
      volume={365},
      number={5},
       pages={2367\ndash 2407},
  url={https://doi-org.ezproxy.library.wisc.edu/10.1090/S0002-9947-2012-05685-6},
      review={\MR{3020102}},
}

\bib{Demeter2020}{book}{
      author={Demeter, C.},
       title={Fourier restriction, decoupling, and applications},
      series={Cambridge Studies in Advanced Mathematics},
   publisher={Cambridge University Press, Cambridge},
        date={2020},
      volume={184},
        ISBN={978-1-108-49970-5},
         url={https://doi.org/10.1017/9781108584401},
      review={\MR{3971577}},
}

\bib{DZ2019}{article}{
      author={Dendrinos, S.},
      author={Zimmermann, E.},
       title={On {$L^p$}-improving for averages associated to mixed homogeneous
  polynomial hypersurfaces in {$\Bbb R^3$}},
        date={2019},
        ISSN={0021-7670},
     journal={J. Anal. Math.},
      volume={138},
      number={2},
       pages={563\ndash 595},
         url={https://doi.org/10.1007/s11854-019-0032-8},
      review={\MR{3996049}},
}

\bib{Docarmo2016}{book}{
      author={do~Carmo, M.~P.},
       title={Differential geometry of curves \& surfaces},
   publisher={Dover Publications, Inc., Mineola, NY},
        date={2016},
        ISBN={978-0-486-80699-0; 0-486-80699-5},
        note={Revised \& updated second edition of [ MR0394451]},
      review={\MR{3837152}},
}

\bib{IM2011}{article}{
      author={Ikromov, I.~A.},
      author={M\"{u}ller, D.},
       title={On adapted coordinate systems},
        date={2011},
        ISSN={0002-9947},
     journal={Trans. Amer. Math. Soc.},
      volume={363},
      number={6},
       pages={2821\ndash 2848},
         url={https://doi.org/10.1090/S0002-9947-2011-04951-2},
      review={\MR{2775788}},
}

\bib{IM2016}{book}{
      author={Ikromov, I.~A.},
      author={M\"{u}ller, D.},
       title={Fourier restriction for hypersurfaces in three dimensions and
  {N}ewton polyhedra},
      series={Annals of Mathematics Studies},
   publisher={Princeton University Press, Princeton, NJ},
        date={2016},
      volume={194},
        ISBN={978-0-691-17055-8},
         url={https://doi-org.ezproxy.library.wisc.edu/10.1515/9781400881246},
      review={\MR{3524103}},
}

\bib{Kemp2}{article}{
      author={Kemp, D.},
       title={Decoupling via scale-based approximations of limited efficacy},
        date={2021},
     journal={preprint},
        note={arXiv:2104.04115},
}

\bib{Kemp}{article}{
      author={Kemp, D.},
       title={Decouplings for surfaces of zero curvature},
        date={2019},
     journal={preprint},
        note={arXiv:1908.07002},
}

\bib{M2014}{incollection}{
      author={M\"{u}ller, D.},
       title={Problems of harmonic analysis related to finite-type
  hypersurfaces in {$\Bbb R^3$}, and {N}ewton polyhedra},
        date={2014},
   booktitle={Advances in analysis: the legacy of {E}lias {M}. {S}tein},
      series={Princeton Math. Ser.},
      volume={50},
   publisher={Princeton Univ. Press, Princeton, NJ},
       pages={303\ndash 345},
      review={\MR{3329856}},
}

\bib{Oberlin2012}{article}{
      author={Oberlin, D.~M.},
       title={A uniform {F}ourier restriction theorem for surfaces in
  {$\Bbb{R}^{d}$}},
        date={2012},
        ISSN={0002-9939},
     journal={Proc. Amer. Math. Soc.},
      volume={140},
      number={1},
       pages={263\ndash 265},
  url={https://doi-org.ezproxy.library.wisc.edu/10.1090/S0002-9939-2011-11218-8},
      review={\MR{2833538}},
}

\bib{PS1997}{article}{
      author={Phong, D.~H.},
      author={Stein, E.~M.},
       title={The {N}ewton polyhedron and oscillatory integral operators},
        date={1997},
        ISSN={0001-5962},
     journal={Acta Math.},
      volume={179},
      number={1},
       pages={105\ndash 152},
         url={https://doi-org.ezproxy.library.wisc.edu/10.1007/BF02392721},
      review={\MR{1484770}},
}

\bib{PS2007}{article}{
      author={Pramanik, M.},
      author={Seeger, A.},
       title={{$L^p$} regularity of averages over curves and bounds for
  associated maximal operators},
        date={2007},
        ISSN={0002-9327},
     journal={Amer. J. Math.},
      volume={129},
      number={1},
       pages={61\ndash 103},
         url={https://doi.org/10.1353/ajm.2007.0003},
      review={\MR{2288738}},
}

\bib{SS2019}{article}{
      author={Schwend, J.},
      author={Stovall, B.},
       title={Fourier restriction above rectangles},
        date={2019},
     journal={preprint},
        note={arXiv:1911.11600},
}

\bib{Yang2}{article}{
      author={Yang, T.},
       title={Uniform $l^2$-decoupling in {$\mathbb{R}^2$} for polynomials},
        date={2021},
     journal={J. Geom. Anal.},
}

\end{biblist}
\end{bibdiv}

\end{document}